\documentclass[12pt]{amsart}
\usepackage{amsmath,amsfonts,amssymb}
\usepackage{aliascnt}
\usepackage[retainorgcmds]{IEEEtrantools}
\usepackage{aliascnt}
\usepackage{graphicx}
\usepackage{mathrsfs}
\usepackage{amscd}

\usepackage[bookmarks=true,pdfstartview=FitH, pdfborder={0 0 0}, colorlinks=true,citecolor=blue, linkcolor=blue]{hyperref}

\newtheorem{theorem}{Theorem}[section]
\newaliascnt{conj}{theorem}
\newaliascnt{cor}{theorem}
\newaliascnt{lemma}{theorem}
\newaliascnt{fact}{theorem}
\newaliascnt{claim}{theorem}
\newaliascnt{prop}{theorem}
\newaliascnt{definition}{theorem}
\newaliascnt{qn}{theorem}
\newaliascnt{assump}{theorem}

\newtheorem{conj}[conj]{Conjecture}
\newtheorem{cor}[cor]{Corollary}
\newtheorem{lemma}[lemma]{Lemma}

\newtheorem{prop}[prop]{Proposition}

\newtheorem{definition}[definition]{Definition}

\aliascntresetthe{conj} \aliascntresetthe{cor}
\aliascntresetthe{lemma} \aliascntresetthe{fact}
\aliascntresetthe{claim} \aliascntresetthe{prop}
\aliascntresetthe{definition} \aliascntresetthe{qn}
\aliascntresetthe{assump}

\theoremstyle{definition}
\newaliascnt{example}{theorem}
\newtheorem{example}[example]{Example}
\aliascntresetthe{example}

\theoremstyle{remark}
\newaliascnt{rmk}{theorem}
\newtheorem{remark}[rmk]{Remark}
\aliascntresetthe{rmk} 

\def\sek~{\S{}}

\DeclareMathOperator{\vspan}{span}
\DeclareMathOperator{\image}{Im}
\DeclareMathOperator{\sech}{sech}

\newcommand{\IsoTo}{\xrightarrow{\raisebox{0.01 mm}{\smash{\ensuremath{\sim}}}}{}}
\newcommand{\p}{\partial}
\newcommand{\F}{\mathcal{F}}
\newcommand{\G}{\mathcal{G}}
\newcommand{\U}{\mathcal{U}}

\newcommand{\RR}{\mathbb{R}}

\makeatletter
\def\blfootnote{\xdef\@thefnmark{}\@footnotetext}
\makeatother

\begin{document}
\title[On Legendrian foliations in contact manifolds II]{On Legendrian foliations in contact manifolds II: Deformation theory}
\author{Yang Huang}

\blfootnote{Y.H was supported by the Center of Excellence Grant ``Centre for Quantum Geometry of Moduli Spaces'' from the Danish National Research Foundation (DNRF95). The author thanks E. Ghys and T. Tsuboi for helpful correspondence.}

\begin{abstract}
Using the structural theorems developed in \cite{H2013}, we study the deformation theory of coisotropic submanifolds in contact manifolds, under the assumption that the characteristic foliation is nonsingular. In the ``middle'' dimensions, we find an interesting relationship with foliation theory. Some elementary applications of contact geometry in foliation theory are explored.
\end{abstract}

\maketitle

\tableofcontents

In a previous work \cite{H2013}, we studied some basic properties of $(n+1)$-dimensional coisotropic submanifolds in $(2n+1)$-dimensional contact manifolds for any $n \geq 1$. In particular, we studied in detail the characteristic foliation $\F$, also known as the Legendrian foliation, on the coisotropic submanifold, which is singular in general. A particularly interesting question is to understand the singularity of $\F$. This was partially done in author's work \cite{H2013,H2014}. 

The current paper is, however, focused on nonsingular Legendrian foliations, and their coisotropic deformations. The deformation problem of coisotropic submanifolds is well-studied in the symplectic case by the work of Oh and Park \cite{OP2005}. In this case, the characteristic foliation, defined to be the kernel of the pre-symplectic form, is automatically nonsingular. Moreover, it is realized in \cite{OP2005} that the deformation theory is governed by an $L_\infty$-algebra, whose formal solution can be obtained as a formal power series. In the contact case, however, we will see in \autoref{subsec:foliation} and \autoref{subsec:gauge_equiv} that the deformation problem of coisotropic submanifolds $Y^{n+1} \subset (M^{2n+1},\xi)$ turns out to be equivalent to the deformation theory of the underlying codimension one foliation. The latter problem is much less algebraic from the author's point of view. Nevertheless, the deformation theory of foliations are studied by several authors, e.g., Hamilton \cite{Ham1978}, Heitsch \cite{He1975} and El Kacimi-Alaoui and Nicolau \cite{EKN1993}. In particular, we will see in \autoref{subsec:gauge_equiv} that the infinitesimal deformation of nonsingular Legendrian foliations is completely characterized by a twisted version of the usual tangential (de Rham) cohomology. This recovers a result Heitsch, which is formulated in a slightly different language.

The reason that we restrict ourself to the case of nonsingular Legendrian foliations is two-fold. First of all, unless we are working with surfaces in contact three-manifolds, it is currently not clear to author what a generic singularity should look like, and in fact, one may not want to work with generic singularities as their dynamics are usually complicated. Secondly and more surprisingly, since the deformation theory of nonsingular Legendrian foliations is equivalent to the deformation theory of codimension one foliations as mentioned above, there are powerful tools in contact geometry, such as Floer theory, which can be effectively applied to the study of foliations. However, in this paper, we will only carry out some rather elementary applications of contact geometry in foliation theory in \autoref{subsec:appl}.  More precisely, we will give a criterion for when a homotopy of foliations is an isotopy, and conjecture that our criterion is necessary and sufficient. The Floer-theoretic approach to foliation theory will be carried out in a separate work.

In the second half of this paper, we work out the deformation theory for general coisotropic submanifolds, again, with the assumption that the characteristic foliation is nonsingular. As in the Legendrian foliation case, it turns out to be closely related to, but not necessarily equivalent to, the deformation problem of the so-called pre-contact structures, which are natural intermediate objects between contact structures and codimension one foliations. See \autoref{subsec:precontact} for more details. We notice that even with the assumption that the characteristic foliation is nonsingular, the general case is still much more complicated than the case of Legendrian foliations. This is largely due to the more involved ``transverse geoemtry'' of the characteristic foliation. Also note that in a recent work of L\^e, Oh, Tortorella and Vitagliano \cite{LOTV}, the authors constructed an $L_\infty$ structure on pre-contact manifolds.

We conclude the paper by giving an explicit calculation of the twisted tangential cohomology for a one-dimensional foliation on the torus in the appendix. It turns out to be very different from the usual tangential cohomology, but so far we do not have an effective way to study such cohomology groups in general.

\section{Nonsingular Legendrian foliations} \label{sec:nonsing_Leg}

Let $(M^{2n+1},\xi)$ be a contact manifold with contact form $\xi=\ker\alpha$. Suppose $Y^{n+1} \subset (M^{2n+1},\xi)$ is a closed, orientable coisotropic $(n+1)$-submanifold with {\em characteristic foliation} $\F$, where $\F=\ker\lambda$ where $\lambda=\alpha|_Y$. We will also call such $\F$ a {\em Legendrian foliation} since the (nonsingular) leaves of $\F$ are Legendrian submanifolds of $M$. We will assume for the rest of this section that $\F$ is a nonsingular codimension one foliation. 

By the standard neighborhood theorem in \cite[Theorem 1.4]{H2013}, a sufficiently small neighborhood of $Y$ in $M$ is contactomorphic to a neighborhood of the zero section of the vector bundle
\begin{equation*}
	\pi: T^\ast \F \to Y,
\end{equation*}
where the fiber at any $y \in Y$ is defined to be $T_y^\ast \F$. Moreover if we fix a transverse line field $L \pitchfork \F$, then the total space of $T^\ast\F$ is equipped with a contact form
\begin{equation} \label{eqn:cntform_nonsing}
	\alpha=\pi^\ast \lambda-\eta,
\end{equation}
where $\eta$ is similar to the tautological one-form on the cotangent bundle, but in this case it depends on our choice of $L$. See \cite{H2013} or \autoref{sec:contact_MEQ} below for the precise definition. Following \cite{LOTV}, we will call this construction a {\em contact thickening}. The main question of interest here is that what is the moduli space of coisotropic submanifolds $C^1$-close\footnote{In fact, we only need the nearby coisotropic subamnifolds to be graphical, which is of course guaranteed by $C^1$-closeness.} to $Y$ in $(M,\xi)$. In the light of the above discussions on the standard contact neighborhood of $Y$, it suffices to determine all $C^1$-small sections $s: Y \to T^\ast\F$ such that the graph of $s$ is coisotropic.

\subsection{The contact master equation for nonsingular Legendrian foliations} \label{sec:contact_MEQ}
To start with, let us pick a foliated chart in $Y$ adapted to $\F$ as follows $$R^{n+1}_{x,q^1,\cdots,q^n} \cong \U \subset Y,$$ where $\U$ is an open neighborhood of some point in $Y$, $(x,q^1,\cdots,q^n)$ are the coordinates on $\U$ and the leaves of $\F$ are identified with $\{x=\text{const}\}$. To write down the contact form (\ref{eqn:cntform_nonsing}) in local coordinates, let $(p_1,\cdots,p_n)$ be the dual coordinates on the fiber of $T^\ast \F$, and let $\lambda=fdx$ for some $f > 0\in C^\infty(\U)$.

In the following we will always have the Einstein's summation convention turned on, i.e., the same upper and lower indices are to be summed up. To define $\eta$, we write the transverse line field $L$ in coordinates as follows
\begin{equation*}
	L=\langle \p_x+R^i\p_{q^i} \rangle,
\end{equation*}
where $R^i\in C^\infty(\U)$. Then we define 
\begin{equation*}
	\eta := -R^ip_idx+p_idq^i 
\end{equation*}
and therefore (\ref{eqn:cntform_nonsing}) can be written as $\alpha=(f+R^ip_i)dx-p_idq^i$. For later use, let us compute
\begin{equation*}
	d\alpha=(\frac{\p f}{\p q^j}+p_k\frac{\p R^k}{\p q^j})dq^j \wedge dx+R^i dp_i \wedge dx-dp_i \wedge dq^i,
\end{equation*}
and so
\begin{IEEEeqnarray*}{rCl}
	\alpha\wedge d\alpha &=& -(f+R^kp_k)dx \wedge dp_i \wedge dq^i-\sum_{i \neq j} p_i(\frac{\p f}{\p q^j}+p_k\frac{\p R^k}{\p q^j})dx \wedge dq^i \wedge dq^j \nonumber\\
					   && +\: p_jR^idx \wedge dp_i \wedge dq^j+p_idq^i \wedge dp_j \wedge dq^j \\
					   &=& -(f+R^kp_k)dx \wedge dp_i \wedge dq^i-\sum_{i<j}\Big(p_i(\frac{\p f}{\p q^j}+p_k\frac{\p R^k}{\p q^j})-p_j(\frac{\p f}{\p q^i} \nonumber\\
					   && +\: p_k\frac{\p R^k}{\p q^i})\Big)dx \wedge dq^i \wedge dq^j+p_jR^idx \wedge dp_i \wedge dq^j+p_idq^i \wedge dp_j \wedge dq^j. 
\end{IEEEeqnarray*}

For any section $s: Y \to T^\ast\F$, denote by $Y_s$ the graph of $s$. Using local coordinates as above, we can write $s$ as a set of functions $p_i=s_i(x,q^1,\cdots, q^n)$. Then it is easy to see that
\begin{equation*}
	TY_s=\langle \p_x+\frac{\p s_i}{\p x}\p_{p_i}, \p_{q^1}+\frac{\p s_i}{\p q^1}\p_{p_i}, \cdots, \p_{q^n}+\frac{\p s_i}{\p q^n}\p_{p_i} \rangle.
\end{equation*}

Now the condition of $Y_s$ being coisotropic reads
\begin{equation*} \label{eqn:coiso_cond}
	\alpha\wedge d\alpha|_{TY_s}=0.
\end{equation*}
For simplicity of notations, let us write $v_0=\p_x+\frac{\p s_i}{\p x}\p_{p_i}$ and $v_k=\p_{q^k}+\frac{\p s_i}{\p q^k}\p_{p_i}$ for $1 \leq k \leq n$. Then we have for any $1 \leq a<b \leq n$,
\begin{IEEEeqnarray*}{rCl} \label{eqn:v_0ab}
	\alpha \wedge d\alpha(v_0, v_a,v_b) &=& -(f+R^k s_k)(\frac{\p s_b}{\p q^a}-\frac{\p s_a}{\p q^b})-s_a(\frac{\p f}{\p q^b}+s_k\frac{\p R^k}{\p q^b})+s_b(\frac{\p f}{\p q^a}+s_k\frac{\p R^k}{\p q^a}) \nonumber\\
		&& +\: s_bR^i\frac{\p s_i}{\p q^a}-s_aR^i\frac{\p s_i}{\p q^b}+s_b\frac{\p s_a}{\p x}-s_a\frac{\p s_b}{\p x} \nonumber\\
		&=& \Big( -f(\frac{\p s_b}{\p q^a}-\frac{\p s_a}{\p q^b})-s_a\frac{\p f}{\p q_b}+s_b\frac{\p f}{\p q^a} \Big)+\Big( s_b\frac{\p (s_kR^k)}{\p q^a}-s_k R^k\frac{\p s_b}{\p q^a}+s_b\frac{\p s_a}{\p x} \Big) \nonumber\\
		&& -\: \Big( s_a\frac{\p (s_k R^k)}{\p q^b}-s_k R^k\frac{\p s_a}{\p q^b}+s_a\frac{\p s_b}{\p x} \Big),
\end{IEEEeqnarray*}
and for any $1 \leq a < b < c \leq n$,
\begin{equation*}
\alpha \wedge d\alpha(v_a,v_b,v_c)=-s_a\frac{\p s_b}{\p q^c}-s_b\frac{\p s_c}{\p q^a}-s_c\frac{\p s_a}{\p q^b}+s_a\frac{\p s_c}{\p q^b}+s_b\frac{\p s_a}{\p q^c}+s_c\frac{\p s_b}{\p q^a}.
\end{equation*}

To summarize the above calculations, we have proved the following result.

\begin{prop} \label{thm:master}
The graph of a section $s: Y \to T^\ast\F$ is coisotropic with respect to $\alpha$ if and only if the following two equations hold.
	\begin{equation} \label{eqn:master1}
		f(\frac{\p s_a}{\p q^b}-\frac{\p s_b}{\p q^a})-s_a\frac{\p f}{\p p_b}+s_b\frac{\p f}{\p q^a} = s_a\frac{\p (s_kR^k)}{\p q^b}-s_kR^k\frac{\p s_a}{\p q^b}+s_a\frac{\p s_b}{\p x} - \Big( s_b\frac{\p (s_kR^k)}{\p q^a}-s_kR^k\frac{\p s_b}{\p q^a}+s_b\frac{\p s_a}{\p x} \Big)
	\end{equation}
for any $1 \leq a<b \leq n$, and
	\begin{equation} \label{eqn:master2}
		-s_a\frac{\p s_b}{\p q^c}-s_b\frac{\p s_c}{\p q^a}-s_c\frac{\p s_a}{\p q^b}+s_a\frac{\p s_c}{\p q^b}+s_b\frac{\p s_a}{\p q^c}+s_c\frac{\p s_b}{\p q^a}=0
	\end{equation}
for any $1 \leq a<b < c \leq n$.
\end{prop}

We will call (\ref{eqn:master1}) and (\ref{eqn:master2}) the {\em contact master equations} for Legendrian foliations since they govern, at least locally, the deformation of coisotropic submanifolds. In fact we will work out the more general contact master equations for coisotropic submanifolds of any dimension in \autoref{sec:generalcoiso}.

Our next goal is to get an invariant understanding of (\ref{eqn:master1}) and (\ref{eqn:master2}). In order to do this, we need some preparations on {\em tangential de Rham theory}, which is well-known in foliation theory. Given a foliation $\F$, we define
\begin{equation*}
	\Omega^\bullet(\F)=\Gamma(\wedge^\bullet T^\ast\F)
\end{equation*}
to be the complex of differential forms on $\F$, where the differential
\begin{equation*}
	d_\F: \Omega^\bullet(\F) \to \Omega^{\bullet+1}(\F)
\end{equation*}
is defined by Cartan's formula as follows.
\begin{IEEEeqnarray*}{rCl}
	d_\F\omega(X_1,\cdots,X_{k+1}) &=& \sum_i (-1)^{i+1}X_i(\omega(X_1,\cdots,\widehat{X_i},\cdots,X_{k+1})) \\
		&& +\: \sum_{i<j} (-1)^{i+j}\omega([X_i,X_j],X_1,\cdots,\widehat{X_i},\cdots,\widehat{X_j},\cdots,X_{k+1}),
\end{IEEEeqnarray*}
for any $\omega\in \Omega^k(\F)$ and any vector fields $X_1,\cdots,X_{k+1}$ tangent to $\F$. Here $\widehat{X_i}$ means the term with $X_i$ is omitted. It is easy to see that $d_\F$ is well-defined since $\F$ is integrable. We now define the {\em tangential de Rham cohomology} $$H^\bullet(\F) := H^\bullet(\Omega^\bullet(\F),d_\F).$$

Let $L$ be a line field transverse to $\F$ as before. We define an inclusion map $i_L: \Omega^\bullet(\F) \to \Omega^\bullet(Y)$ by requiring $i_L(\omega)$ to be equal to $\omega$ in the $\F$-directions and to vanish in the $L$-direction, for any $\omega \in \Omega^\bullet(\F)$. We will write down a local expression of $i_L(\omega)$ in coordinates introduced above for $\omega \in \Omega^1(\F)$, as it is enough for our later purposes. Namely, let $\omega=a_i dq^i \in \Omega^1(\F)$, we have
\begin{equation} \label{eqn:bar}
	\overline\omega := i_L(\omega)=-R^j a_j dx+a_i dq^i \in \Omega^1(Y).
\end{equation}
We also have a natural projection map $\pi: \Omega^\bullet(Y) \to \Omega^\bullet(\F)$ simply by restriction. Of course we have $\pi(\overline\omega)=\omega$ for any $\omega \in \Omega^\bullet(\F)$. But $i_L \circ \pi$ is in general not the identity map. Nevertheless we have $\overline{\pi(\sigma)}=\sigma$ for any $\sigma \in \Omega^\bullet(Y)$ such that $\sigma(L)=0$.

Now we can reformulate \autoref{thm:master} in a coordinate-free manner as follows.

\begin{theorem}
The graph of a section $s: Y \to T^\ast\F$, viewed as a one-form $s \in \Omega^1(\F)$, is coisotropic with respect to $\alpha$ if and only if the following two equations hold.
	\begin{align}
		\overline{s} \wedge d \overline{s} &= \overline{s} \wedge d\lambda +d\overline{s} \wedge \lambda  \hspace{3mm} \text{on} \hspace{2mm}  \Omega^3(Y), ~~~\text{and}\label{eqn:supfol}\\
		s \wedge d_\F s &= 0 \hspace{3mm} \text{on} \hspace{2mm} \Omega^3(\F),  \label{eqn:subfol}
	\end{align}
where $\overline{s}$ is defined by (\ref{eqn:bar}).
\end{theorem}

\begin{proof}
Using a local foliated chart $\U \cong \RR^{n+1}_{x,q^1,\cdots,q^n}$ as before, we write $s=s_i dq^i$, where each $s_i$ is a smooth function on $\U$. Then (\ref{eqn:subfol}) is equivalent to the following.
	\begin{equation} \label{eqn:subfol_calc}
		s \wedge d_\F s = s_i dq^i \wedge \Big( \frac{\p s_k}{\p q^j} dq^j \wedge dq^k \Big) = s_i\frac{\p s_k}{\p q^j} dq^i \wedge dq^j \wedge dq^k=0,
	\end{equation}
which is precisely (\ref{eqn:master2}). Next we compute
	\begin{align}
		\overline{s} \wedge d \overline{s} &= \Big( -R^j s_j dx+s_i dq^i \Big) \wedge \Big( \frac{\p (R^k s_k)}{\p q^j} dx \wedge dq^j+\frac{\p s_i}{\p x} dx \wedge dq^i+\frac{\p s_j}{\p q^i} dq^i \wedge dq^j \Big) \nonumber\\
						     &= \Big( R^k s_k \frac{\p s_i}{\p q^j}-s_i\frac{\p (R^k s_k)}{\p q^j}+s_j \frac{\p s_i}{\p x} \Big)dx \wedge dq^i \wedge dq^j+s_i \frac{\p s_j}{\p q^k} dq^i \wedge dq^j \wedge dq^k \nonumber\\
						     &= \Big( R^k s_k \frac{\p s_i}{\p q^j}-s_i\frac{\p (R^k s_k)}{\p q^j}+s_j \frac{\p s_i}{\p x} \Big)dx \wedge dq^i \wedge dq^j, \label{eqn:master_calc1}
	\end{align}
where the last equality follows from (\ref{eqn:subfol_calc}). Finally we compute
	\begin{equation} \label{eqn:master_calc2}
		\overline{s} \wedge d\lambda+d\overline{s} \wedge \lambda= (f \frac{\p s_j}{\p q^i}-s_j\frac{\p f}{\p q^i}) dx \wedge dq^i \wedge dq^j,
	\end{equation}
where $\lambda=f dx$ as before. Now one immediately sees that equating (\ref{eqn:master_calc1}) and (\ref{eqn:master_calc2}) is equivalent to (\ref{eqn:master1}). This finishes the proof.
\end{proof}

We will also refer to (\ref{eqn:supfol}) and (\ref{eqn:subfol}) as the contact master equations.

\subsection{Infinitesimal deformation of Legendrian foliations} \label{sec:leg_infi}
As the first step towards understanding the geometric meaning of (\ref{eqn:supfol}) and (\ref{eqn:subfol}), we will first look at its linearized version, which precisely characterizes the infinitesimal deformations of the Legendrian foliation.

Let us start with a general discussion of a twisted version of the tangential de Rham cohomology associated to a codimension one foliation. More discussions of the distinction between the usual tangential de Rham theory and the twisted version will be carried out in \autoref{apx:twisted_de_rham}. To the best knowledge of the author, this version of tangential cohomology is new to the literature. In fact we will generalize the construction in this section to a class of higher codimension foliations in \autoref{sec:coiso_infi}.

Consider a nonsingular codimension one foliation $\F$ on $Y$, and let $\lambda$ be a defining one-form such that $\F=\ker\lambda$. By Frobenius theorem, we have $\lambda \wedge d\lambda=0$, which in turn implies that there exists $\mu \in \Omega^1(Y)$ such that $d\lambda=\mu \wedge \lambda$. Although $\mu$ is not uniquely determined by $\lambda$, it is easy to see that $\pi(\mu) \in \Omega^1(\F)$ is uniquely determined by $\lambda$. The following lemma is well-known in foliation theory.

\begin{lemma} \label{lem:mu_closed}
The tangential one-form $\pi(\mu)$ is closed, i.e., $d_\F(\pi(\mu))=0$.
\end{lemma}

\begin{proof}
The proof follows from a local calculation. Choosing a foliated chart $R^{n+1}_{x,q^1,\cdots,q^n} \cong \U \subset Y$ as before such that locally $$\lambda=fdx \hspace{5mm} \text{for some} \hspace{2mm} f>0 \in C^\infty (\U). $$ Then we compute
\begin{equation*}
	d\lambda= \frac{\p f}{\p q^i}dq^i \wedge dx=\frac{1}{f} \frac{\p f}{\p q^i}dq^i \wedge \lambda.
\end{equation*}
Therefore in a coordinate chart we have $$\pi(\mu)=\frac{1}{f} \frac{\p f}{\p q^i}dq^i=d_\F (\log f) \in \Omega^1(\F).$$ This proves the conclusion in the lemma since $d_\F$ is a differential.
\end{proof}

\begin{definition}
The {\em twisted tangential differential} $d_\F^\lambda$ on $\Omega^\bullet(\F)$ is defined by $$d_\F^\lambda(\omega)=d_\F \omega-\pi(\mu)\wedge\omega$$ for any $\omega \in \Omega^\bullet(\F)$. The {\em twisted tangential (de Rham) cohomology} $H_{tw}^\bullet(\F)$ is defined to be $H^\bullet(\Omega^\bullet(\F),d_\F^\lambda)$.
\end{definition}

The twisted tangential differential and the twisted tangential cohomology are well-defined by the following two lemmas.

\begin{lemma}
$d_\F^\lambda \circ d_\F^\lambda=0$.
\end{lemma}

\begin{proof}
It follows from the following straightforward calculation. For any $\omega \in \Omega^\bullet(\F)$,
	\begin{align*}
		d_\F^\lambda \circ d_\F^\lambda(\omega) &= d_\F^\lambda(d_\F \omega-\pi(\mu)\wedge\omega) \\
			&= d_\F^2 \omega-d_\F(\pi(\mu)\wedge\omega)-\pi(\mu)\wedge(d_\F \omega)+\pi(\mu)\wedge\pi(\mu)\wedge\omega \\
			&=-d_\F(\pi(\mu))\wedge\omega=0.
	\end{align*}
Here we used the fact that $\pi(\mu)$ is closed by \autoref{lem:mu_closed}.
\end{proof}

\begin{lemma} \label{lem:well-defined-tangential}
The isomorphism class of $H^\bullet_{tw}(\F)$ does not depend on the choice the defining one-form $\lambda$, which justifies the notation.
\end{lemma}

\begin{proof}
The proof is divided into two steps. First we will show that the cohomology class $[\pi(\mu)] \in H^1(\F)$ is independent of the choice of $\lambda$. Indeed, suppose $\lambda'=g\lambda$ is another defining one-form of $\F$, where $g>0 \in C^\infty(Y)$. Similarly let $\mu'$ be such that $d\lambda'=\mu' \wedge\lambda'$. Then we have
	\begin{align*}
		d\lambda' &= dg \wedge\lambda+g d\lambda \\
				&= (dg+g\mu)\wedge\lambda \\
				&= (d(\log g)+\mu)\wedge\lambda',
	\end{align*}
therefore we may choose $\mu'=\mu+d(\log g)$. Hence $\pi(\mu')=\pi(\mu)+d_\F(\log g)$, which implies that the class $[\pi(\mu)]$ depends only on $\F$ rather than $\lambda$.

Next we will show that the quasi-isomorphism class of the chain complex $(\Omega^\bullet(\F),d_\F^\lambda)$ depends only on $[\pi(\mu)] \in H^1(\F)$, which immediately implies the lemma. Let $d_\F^{\lambda'}$ be another twisted differential defined by $d_\F^{\lambda'}(\omega)=d_\F \omega-(\pi(\mu)+d_\F h)\wedge\omega$ for some $h \in C^\infty(Y)$ and any $\omega \in \Omega^\bullet(\F)$. We now define a chain map $\Phi$ in the following diagram
\begin{center}
$\begin{CD}
\cdots @>>> \Omega^k(\F) @>d_\F^\lambda>> \Omega^{k+1}(\F) @>>> \cdots\\
@. @VV\Phi V @VV\Phi V\\
\cdots @>>> \Omega^k(\F) @>d_\F^{\lambda'}>> \Omega^{k+1(\F)} @>>> \cdots
\end{CD}$
\end{center}
by $\Phi(\omega)=e^{h} \omega$ for any $\omega \in \Omega^\bullet(\F)$. We now check that $\Phi$ indeed defines a chain map by computing 
\begin{align*}
	d_\F^{\lambda'}(\Phi(\omega)) &= d_\F^{\lambda'}(e^{h}\omega) \\
		&= d_\F(e^{h}\omega)-(\pi(\mu)+d_\F h) \wedge e^{h}\omega \\
		&= e^{h}(d_\F \omega-\pi(\mu)\wedge\omega)=\Phi(d_\F^\lambda(\omega)).
\end{align*}
One can similarly define the inverse chain map $\Phi^{-1}$ by $\Phi^{-1}(\omega)=e^{-h} \omega$. In particular $\Phi$ induces an isomorphism on cohomology groups, as desired.
\end{proof}

Returning to the deformation theory of Legendrian foliations, we have the following characterization of infinitesimal deformations of Legendrian foliations in terms of one-cocycles in the twisted tangential de Rham cohomology.

\begin{theorem} \label{thm:leg_infi}
A tangential one-form $\zeta \in \Omega^1(\F)$ is an infinitesimal deformation of the (nonsingular) Legendrian foliation $(Y,\F)$ if and only if $d_\F^\lambda \zeta=0$ on $\Omega^2(\F)$.
\end{theorem}

\begin{proof}
Notice that (\ref{eqn:subfol}) and the LHS of (\ref{eqn:supfol}) are both quadratic, so the linearized equation which governs the infinitesimal deformation is just
\begin{equation*}
	\overline\zeta \wedge d\lambda+d\overline\zeta \wedge \lambda=(\overline\zeta \wedge\mu+d\overline\zeta) \wedge \lambda=0
\end{equation*}
on $\Omega^3(Y)$. This is easily seen to be equivalent to the equation
\begin{equation*}
	d_\F^\lambda(\zeta)=d\zeta-\mu\wedge\zeta=0
\end{equation*}
on $\Omega^2(\F)$. This finishes the proof.
\end{proof}

However, to correctly formulate the moduli problem of (infinitesimal) coisotropic deformations, we need to mod out by certain ``gauge equivalence'' relations, which will be discussed in \autoref{subsec:gauge_equiv}. It turns out that, similar to the deformation theory of Lagrangian submanifolds in a symplectic manifold, the moduli space of infinitesimal deformations of a coisotropic $(Y,\F) \subset (M,\xi)$ is governed by $H^1_{tw} (\F)$.

\subsection{Relationship with foliation theory} \label{subsec:foliation}
The goal of this section is to get a better understanding of the contact master equations by relating it to foliations. First we notice that (\ref{eqn:supfol}) is equivalent to the equation $$(\overline{s}-\lambda) \wedge d(\overline{s}-\lambda)=0,$$ since $\lambda \wedge d\lambda=0$ by Frobenius integrability theorem. Next note that (\ref{eqn:subfol}) is equivalent to saying $\ker s \subset \F$ also defines a (singular) foliation of codimension two in $Y$, again by Frobenius theorem. So if we let $s' := \overline{s}-\lambda$, Then $\ker s'$ defines a nonsingular codimension one foliation on $Y$ since $s'(L) \neq 0$. Conversely, given any nonsingular codimension one foliation $\mathcal{G}=\ker \mu$ which is $C^0$-close to $\F$, it is also transverse to $L$. So we may rescale $\mu$ by a nonzero function such that $\lambda(L)=-\mu(L)$, and define $s=\pi(\lambda+\mu) \in \Omega^1(\F)$. Then $s$ obviously satisfies (\ref{eqn:supfol}). To see $s$ also satisfies (\ref{eqn:subfol}), note that the distribution $\ker s \subset \F$ coincides with $\ker \lambda \cap \ker \mu=\F \cap \mathcal{G}$, which is clearly integrable.

In the light of the above discussion, the following result is expected.

\begin{theorem} \label{thm:leg_deform}
Given a coisotropic submanifold $Y \subset (M,\xi)$ with a nonsingular Legendrian foliation $\F$. Fix a line field $L$ transverse to $\F$. There is a one-to-one correspondence between
\begin{center}
	\{coisotropic submanifolds $C^1$-close to $Y$\}
\end{center}
and
\begin{center}
	\{codimension one foliations $C^1$-close to $\F$ in $Y$\}.
\end{center}
Moreover a coisotropic deformation of $Y$ is locally homotopic to $Y$, i.e., there exists a connecting path of coisotropic submanifolds, all of which are contained in a standard neighborhood of $Y$, if and only if their corresponding foliations in $Y$ are locally homotopic, i.e., homotopic within a $C^0$-small neighborhood of $\F$.
\end{theorem}

\begin{proof}
The one-to-one correspondence is clear from the above discussions. In fact any $C^0$-small deformation\footnote{Two codimension one foliations are said to be $C^0$-close if the angle between the two leaves at each point is sufficiently small.} of $\F$ also gives rise to a $C^0$-small deformation of $Y$, but the converse is not necessarily true since a $C^0$-small deformation of $Y$ may not be graphical. So it suffices to prove the last statement. Suppose $Y' \subset M$ is another coisotropic submanifold, $C^1$-close to $Y$, such that $Y'$ is locally homotopic to $Y$, i.e., there is a path of coisotropic manifolds $Y_t, t\in [0,1]$, contained in a standard neighborhood of $Y$ such that $Y_0=Y$ and $Y_1=Y'$. Regarding $Y_t$ as the graph of tangential one-forms as usual, we obtain a path of one-forms $\mu_t, t \in [0,1]$, which defines a path of foliations $\F_t=\ker \mu_t$ in $Y$. This clearly gives the desired homotopy of $\F$. The converse works similarly and is omitted here. 
\end{proof}

\begin{remark}
Since the space of foliations is not locally path-connected, neither is the space of nearby coisotropic submanifolds. This is very different from the case of Legendrian (or Lagrangian) submanifolds, but is of course expected.
\end{remark}

Our next task is to understand the effect of coisotropic deformations on the Legendrian foliation. To this end, we must identify the original coisotropic submanifold with the deformed one. Suppose $Y \subset (M,\xi)$ is coisotropic with Legendrian foliation $\F$. Recall that a $C^1$-small deformation of $Y$ is another coisotropic submanifold $Y'$ given by the graph of a tangential one-form $s \in \Omega^1(\F)$, where $s$ is viewed as a section in the leafwise cotangent bundle $T^\ast\F$. Using local coordinates as before, let us define a diffeomorphism $\Phi_s: T^\ast\F \to T^\ast\F$ by $\Phi(y,\mathbf{q},\mathbf{p}) = (y,\mathbf{q},\mathbf{p}+s(y,\mathbf{q}))$, where $(y,\mathbf{q})$ are the coordinates on the base and $\mathbf{p}$ are the coordinates on the fiber. Then it is clear that $\Phi_s(Y)=Y'$. In the following we will always identify $Y'$ with $Y$ via $\Phi_s$. Now we compute
\begin{align*}
	\Phi_s^\ast(\alpha) &= \Phi_s^\ast ((f+R^\nu p_\nu)dx-p_\mu dq^\mu) \\
				  &= (f+R^\nu (p_\nu+s_\nu))dx-(p_\mu+s_\mu) dq^\mu.
\end{align*}
Hence the new Legendrian foliation $\F'$ on $Y'$ is defined by
\begin{equation*}
	\F'=\ker (\Phi^\ast \alpha |_{\{\mathbf{p}=0\}}) = \ker((f+R^\nu s_\nu)dx-s_\mu dq^\mu)=\ker(\overline{s}-\lambda),
\end{equation*}
where $\overline{s}$ is defined by (\ref{eqn:bar}). So we have proved the following result.

\begin{theorem} \label{thm:leg_real}
Given a coisotropic submanifold $Y \in (M,\xi)$ with a nonsingular Legendrian foliation $\F$. Then any foliation $\F'$ on $Y$, which is $C^0$-close to $\F$, may be realized as the Legendrian foliation of another coisotropic submanifold $Y'$ contained in a standard neighborhood of $Y$.
\end{theorem}

We end this section by a simple example which shows that the coisotropic deformation (in the nonsingular case) is in general obstructed. Following \autoref{thm:leg_deform}, this is the same as showing that the deformation of a nonsingular codimension one foliation is obstructed.

\begin{example}
Consider $T^3_{x,y,z}=\mathbb{R}^3 / \mathbb{Z}^3$ and the foliation $\F=\ker\lambda$ where $\lambda=dz$.  Let $L=\langle \p_z \rangle$ be the transverse line field. Then clearly $d\lambda=0$ and therefore $d_\F=d_\F^\lambda$. It follows from \autoref{thm:leg_infi} that the tangential one-form $\zeta=\cos zdx+\sin zdy$ defines an infinitesimal deformation. Suppose there is a genuine coisotropic deformation $\lambda_t=dz+t\overline{\zeta}+t^2\overline{\eta}+o(t^3)$, namely, $\lambda_t \wedge d\lambda_t=0$. Then by comparing the second-order term, we must have
\begin{equation*}
	d_\F \eta=\zeta \wedge \frac{\p \zeta}{\p z}=dx \wedge dy.
\end{equation*}
But this is clearly impossible. Hence $\F$ cannot be deformed in the direction of $\zeta$.
\end{example}

\subsection{Gauge equivalence} \label{subsec:gauge_equiv}
The goal of this section is to give a strengthened version of \autoref{thm:leg_deform} by taking into account the so-called gauge equivalences. In this way we will also complete the description of the moduli space of (infinitesimal) coisotropic deformations as promised in the end of \autoref{sec:leg_infi}.

Given a coisotropic submanifold with nonsingular Legendrian foliation $(Y,\F) \subset (M,\xi)$, we say two nearby coisotropic submanifolds $Y_0, Y_1$ are {\em gauge equivalent} if there is a $C^1$-small ambient contact isotopy $\phi_t: M \to M, t \in [0,1]$, which is supported in the standard neighborhood of $Y$, and is such that $\phi_0$ is the identity map and $\phi_1(Y_0)=Y_1$.

Similarly, consider two foliations $\F_0,\F_1$ which are $C^0$-close to a fixed foliation $\F$ on $Y$. We say $\F_0$ and $\F_1$ are {\em gauge equivalent} if there is a $C^1$-small ambient isotopy $\psi_t: Y \to Y, t \in [0,1]$, through diffeomorphisms such that $\psi_0$ is the identity map and $\psi_1(\F_0)=\F_1$. Equivalently, suppose $\F_i=\ker(\lambda_i), i=0,1$, for some $\lambda_i \in \Omega^1(Y)$. Then we the condition that $\psi_1(\F_0)=\F_1$ is the same as that $\psi_1^\ast(\lambda_1)=g\lambda_0$ for some non-vanishing $g \in C^\infty(Y)$. Note that unlike contact structures, foliations are unstable structures in the sense that two isotopic foliations are not necessarily gauge equivalent.

The main result of this section is as follows.

\begin{theorem} \label{thm:ModGauge}
Given the same conditions as in \autoref{thm:leg_deform}, there is a one-to-one correspondence between the gauge equivalent classes of coisotropic submanifolds $C^1$-close to $Y$ and the gauge equivalent classes of codimension one foliations $C^1$-close to $\F$.
\end{theorem}

\begin{proof}
It suffices to prove that two coisotropic deformations are gauge equivalent if and only if their corresponding foliations are gauge equivalent. Our proof mimics the proof of the well-known ambient contact isotopy theorem for Legendrian submanifolds. 

First, let $Y_{\mu_0}, Y_{\mu_1}$ be two coisotropic submanifolds in the standard neighborhood of $Y$ given by the graphs of $\mu_0, \mu_1 \in \Omega^1(\F)$, respectively. Suppose $\phi_t: M \to M, t \in [0,1]$, is a $C^1$-small ambient contact isotopy such that $\phi_1(Y_{\mu_0})=Y_{\mu_1}$. Then there exists a path of $\mu_t \in \Omega^1(\F)$ such that $\phi_t(Y_{\mu_0})=Y_{\mu_t}$. To get an isotopy of $Y$, recall the discussions prior to \autoref{thm:leg_real} implies that there is a canonical identification $\Phi_{t}: Y \IsoTo Y_{\mu_t}$. So it is easy to see that 
\begin{equation*}
	\psi_t = \Phi_t^{-1} \circ \phi_t \circ \Phi_0: Y \to Y
\end{equation*}
is an isotopy of $Y$ which sends $\F_0=\ker(\overline{\mu_0}-\lambda)$ to $\F_t=\ker(\overline{\mu_t}-\lambda)$ for $0 \leq t \leq 1$ as desired.

Now we prove the other direction. Namely, let $\G_i=\ker(\overline{\mu_i}-\lambda), i=0,1$, be two foliations on $Y$. Suppose there is a $C^1$-small isotopy $\psi_t: Y \to Y, t\in[0,1]$, such that $\psi_1(\G_0)=\G_1$. Let us first use $\psi_t$ to construct an isotopy of diffeomorphisms $\psi'_t: T^\ast\F \to T^\ast\F$ as follows. For any section $\mu \in \Omega^1(\F)$ of $T^\ast\F$, since we assume $\psi_t$ is $C^1$-small, there exists a unique $g \neq 0 \in C^\infty(Y)$ and $\mu_t \in \Omega^1(\F)$ such that $\psi_t(\overline{\mu}-\lambda)=g_t(\overline{\mu_t}-\lambda)$. Then we simply define $\psi'_t(\mu)=\mu_t$. Now $\psi'_t: T^\ast\F \to T^\ast\F$ is clearly a path of diffeomorphisms, but to get an actual contact isotopy, let us observe that by construction, we have a path of contact structures $(\psi'_t)^\ast(\xi)$ in a neighorbood of $Y_{\mu_0}$ such that $Y_{\mu_0}$ stays coisotropic with Legendrian foliation $\G_0$ throughout $t \in [0,1]$. By the neighborhood theorem in \cite[Lemma 3.2]{H2013}, there exists a path of diffeomorphisms $\rho_t: N(Y_{\mu_0}) \to N(Y_{\mu_0}), t\in[0,1]$, such that $\rho_0$ are $\rho_t |_{Y_{\mu_0}}$ are the identity maps, and $\rho_t^\ast ((\psi'_t)^\ast (\xi))=\xi$ for any $t \in [0,1]$. Finally observe that $\psi''_t=\psi'_t \circ \rho_t$ is a contactomorphisms whenever it is defined, and it is easy to extend it to a everywhere defined contactomorphism which we still denote by $\psi''_t: T^\ast\F \to T^\ast\F$. Then $\psi''_1(Y_{\mu_0})=Y_{\mu_1}$ as desired.
\end{proof}

It follows from \autoref{thm:ModGauge} that the deformation problem of nonsingular Legendrian foliations is the same as the deformation problem of smooth codimension one foliations. The solution to the later is well-known, in a slightly different language, in foliation theory by the work of Heitsch \cite{He1975}. Here we nevertheless formulate the result as follows and provide an alternative proof using Moser's technique.

\begin{cor} \label{cor:infini}
The infinitesimal deformation of a nonsingular Legendrian foliation $(Y,\F)$ modulo gauge equivalence is modeled by $H^1_{tw}(\F)$.
\end{cor}

\begin{proof}
By \autoref{thm:leg_infi} and \autoref{thm:ModGauge}, it suffices to show that a deformation of $\F$ induced by an ambient isotopy $\psi_t: Y \to Y$, with $\psi_0=\text{id}$, is infinitesimally determined by a coboundary in the twisted tangential complex $(\Omega^\bullet (\F), d_\F^\lambda)$, where $\lambda$ is a defining one-form of $\F$.

To this end, let $\lambda_t=\psi_{-t}^\ast(\lambda)$, and write $\lambda_t=g_t(\lambda+t\mu+o(t^2))$ where $g_t$ is a non-vanishing function on $Y$ with $g_0=1$ and $\mu \in \Omega^1(\F)$ is a tangential one-form\footnote{Hereafter we will be too lazy to distinguish $\mu$ and $\overline{\mu}$ if no confusion is possible.}. Then the goal is show that $\mu$ is a coboundary in $\Omega^1(\F)$ with respect to $d_\F^\lambda$. By differentiating with respect to $t$ on both sides of the following equation
\begin{equation*}
	\psi_t^\ast \lambda_t=\lambda,
\end{equation*}
and simplify the result a little bit, we have
\begin{equation*}
	\mathcal{L}_V \lambda_t+\dot\lambda_t=0,
\end{equation*}
where $\psi_t$ is given by the time-$t$ flow of the vector field $V$. 

Recall the splitting $TY=T\F \oplus L$, we can write $V=V_0+V_1$ such that $V_0$ is tangent to $L$ and $V_1$ is tangent to $\F$. Using Cartan's formula and the $t$-expansion of $\lambda_t$, we have the following equation at time $t=0$.
\begin{equation*}
	0 = \mathcal{L}_V \lambda+\dot g_{t=0} \lambda+\mu = d(\lambda(V_0))+(\delta(V_1)+\dot g_{t=0})\lambda-\lambda(V_0)\delta+\mu
\end{equation*}
Finally, by collecting all the tangential terms from the above equation, we see $\mu=-d_\F^\lambda (\lambda(V_0))$ as desired.
\end{proof}

\subsection{More applications in foliation theory} \label{subsec:appl}
From the previous sections it should be clear that there is a close relationship between foliation and coisotropic submanifolds via contact thickening. In fact in the light of \autoref{thm:ModGauge}, one should expect some new invariants of foliations which come from invariants of coisotropic submanifolds of a contact manifold, in particular, Floer-theoretic invariants. However in this paper, we will restrict ourself to more classical tools in contact geometry, which already give some interesting applications to foliation theory.

Given a closed manifold $Y$ with a smooth codimension one foliation $\F$. Let $\F_t, t\in[0,1]$, be a $C^0$-small homotopy\footnote{This is not a serious restriction because any homotopy of foliations can be decomposed into a sequence of $C^0$-small ones.} of foliations with $\F_0=\F$. We will investigate in this section when $\F_t$ is an isotopy of foliations, i.e., when there exists an isotopy $\phi_t:Y \to Y$, such that $\phi_t^\ast(\F_t)=\F$. As usual, we can consider the contact manifold $T^\ast\F$ in which $(Y,\F)$ is a Legendrian foliation as the zero section. Then we can identify the path $\F_t$ of foliations as a path of (graphical) coisotropic submanifolds $Y_t \subset T^\ast \F$ such that $Y=Y_0$. By \autoref{thm:ModGauge}, the homotopy $\F_t$ of foliation is an isotopy if and only if there exists a contact isotopy $\psi_t: T^\ast\F \to T^\ast\F$ such that $\psi_0=id$ and $\psi_t(Y)=Y_t$ for any $t \in [0,1]$. The point is that any contact isotopy is generated by a contact Hamiltonian function, which we briefly recall as follows.

Given a contact manifold $(M,\xi=\ker\alpha)$, and a smooth function $H \in C^\infty(M)$ which we call the contact Hamiltonian function, there is a unique contact vector field $X_H=HR-\nu$, where $R$ is the Reeb vector field and $\nu \in \ker\alpha$ is a vector field such that $i_\nu d\alpha=dH|_\xi$. We call such $X_H$ the contact Hamiltonian vector field associated with $H$. The time-$t$ flow $\phi_t: M \to M$ of $X_H$ is clearly a contact isotopy. It turns out that any contact isotopy is given by the flow of some contact Hamiltonian vector field.

Returning to our situation, choose a foliated chart $\U \cong \RR_x \times \RR_{q^1,\cdots,q^n}$ in $(Y,\F)$ and write the transverse line field $L=\langle \p_x+R^i \p_{q^i} \rangle$ as before. Then the contact form is written as $\alpha=(f+R^i p_i)dx-p_i dq^i$, where $\lambda=\alpha|_Y=fdx$ is the defining one-form of $\F$. The Reeb vector field is easily computed as follows
\begin{equation*}
	R=\big( \p_x+R^i \p_{q^i}-(\frac{\p f}{\p q^i}+p_j \frac{\p R^j}{\p q^i}) \p_{p_i} \big)/f.
\end{equation*}

Now given a (time-dependent) contact Hamiltonian function $H \in C^\infty(T^\ast\F)$, we have, after somewhat tedious calculations, the corresponding contact Hamiltonian vector field as follows
\begin{equation*}
	X_H=\frac{1}{f} \Big( (H-p_k \frac{\p H}{\p p_k})\p_x + (HR^i-f\frac{\p H}{\p p_i}-R^ip_k\frac{\p H}{\p p_k})\p_{q^i}-\big( (H-p_j \frac{\p H}{\p p_j}) (\frac{\p f}{\p q^i}+p_k \frac{\p R^k}{\p q^i})-\Phi p_i-f\frac{\p H}{\p q^i} \big)\p_{p_i} \Big),
\end{equation*}
where $\Phi=\p_x H+R^k\p_{q^k} H-(\p_{q^j} f+p_k \p_{q^j} R^k) \p_{p_j} H$.

Hence the time-$t$ flow $\phi_t$ of $X_H$ is given by the solution to the following system of ODEs.
\begin{equation} \label{eqn:contact_flow}
\begin{cases}
	\dot{x}(t) &= \frac{1}{f} (H-p_k \frac{\p H}{\p p_k}), \\
	\dot{q}^i(t) &= \frac{1}{f} (HR^i-f\frac{\p H}{\p p_i}-R^ip_k\frac{\p H}{\p p_k}), \\
	\dot{p}_i(t) &= \frac{1}{f} (\Phi p_i+f\frac{\p H}{\p q^i}-(H-p_j \frac{\p H}{\p p_j})(\frac{\p f}{\p q^i}+p_k \frac{\p R^k}{\p q^i})).
\end{cases}
\end{equation}

Let us identify the path of isotropic submanifolds $Y_t=\phi_t(Y)$ with the path of foliations $\F_t=\ker\lambda_t$ where
\begin{equation} \label{eqn:lambda_t}
	\lambda_t = (f+R^k p_k(t))dx-p_i(t) dq^i.
\end{equation}
Here $p_i(t)=p_i(t)(x,\mathbf{q})$ is a smooth function on $Y$ for any $i$. By differentiating with respect to $t$ and hiding the $t$-variable for conciseness, we have
\begin{equation} \label{eqn:pdt}
	\dot{p}_i = \frac{d}{dt} (p_i (x,\mathbf{q})) = \dot{p}_i(x,\mathbf{q})+\frac{\p p_i} {\p x} \dot{x}+\frac{\p p_i} {\p q^k} \dot{q}^k,
\end{equation}
for any $i$. Now plugging (\ref{eqn:contact_flow}) into (\ref{eqn:pdt}), we have
\begin{align} \label{eqn:isotopy_condition}
	\dot{p}_i = ~&\frac{1}{f}(\frac{\p p_i}{\p x}+\frac{\p f}{\p q^i}+p_j \frac{\p R^j}{\p q^i}+R^k \frac{\p p_i}{\p q^k}) (H-p_k \frac{\p H}{\p p_k}) - \frac{\p p_i}{\p q^k} \frac{\p H}{\p p_k} \\
				 &- \frac{\p H}{\p q^i}-\frac{1}{f}(\frac{\p H}{\p x}+R^k \frac{\p H}{\p q^k}-(\frac{\p f}{\p q^j}+p_k \frac{\p R^k}{\p q^j}) \frac{\p H}{\p p_j}) p_i. \nonumber
\end{align}

\begin{remark}
At time $t=0$, we have $p_{i,t=0}=0$ for all $i$ by assumption. So (\ref{eqn:isotopy_condition}) takes the following simple form
\begin{equation*}
	\dot{p}_{i,t=0} = \frac{H_{t=0}}{f} \frac{\p f}{\p q^i}-\frac{\p H_{t=0}}{\p q^i} = -d_{\F}^\lambda H_{t=0}.
\end{equation*}
This, of course, coincides with the conclusion in \autoref{cor:infini}.
\end{remark}

In principle, it should give a complete characterization of when a homotopy of foliations is an isotopy by investigating the flow of $X_H$ for all $H$. But in practice, it is helpful to build our way up by setting up a hierarchy on the complexity of $H$, say, by the homogenous degree of $H$ in the cotangent directions. This is what we will do in the following.

\subsubsection{Zeroth-order approximation} \label{subsec:0order}
Let us first consider the time-dependent contact Hamiltonian $H \in C^\infty(Y)$, i.e., $H$ does not depend on the cotangent directions. Then (\ref{eqn:isotopy_condition}) can be simplified as follows.
\begin{equation} \label{eqn:0orderP}
	\dot{p}_i=(\frac{\p p_i}{\p x}+\frac{\p f}{\p q^i}+p_k \frac{\p R^k}{\p q^i}+R^k \frac{\p p_i}{\p q^k}) \frac{H}{f}-\frac{\p H}{\p q^i}-(\frac{\p H}{\p x}+R^k \frac{\p H}{\p q^k}) \frac{p_i}{f}
\end{equation}

By plugging (\ref{eqn:0orderP}) into (\ref{eqn:lambda_t}), we have
\begin{align} \label{eqn:isotopy_notime}
	\dot{\lambda} &= (R^k \dot{p}_k) dx-\dot{p}_i dq^i \\
				  &= \frac{H}{f} \Big( R^j (\frac{\p p_j}{\p x}+\frac{\p f}{\p q^j}+p_k \frac{\p R^k}{\p q^j}+R^k \frac{\p p_j}{\p q^k}) dx-(\frac{\p p_i}{\p x}+\frac{\p f}{\p q^i}+p_k \frac{\p R^k}{\p q^i}+R^k \frac{\p p_i}{\p q^k}) dq^i \Big) \nonumber\\
				  &\hspace{5mm} +\frac{\p H}{\p q^i} dq^i-R^k \frac{\p H}{\p q^k} dx+\frac{1}{f} (\frac{\p H}{\p x}+R^k \frac{\p H}{\p q^k}) (p_i dq^i-R^k p_k dx) \nonumber
\end{align}

To get a more invariant form of (\ref{eqn:isotopy_notime}), let $v$ be the vector field tangent to $L$ such that $\lambda(v)=1$. Recall $\F_0=\F=\ker\lambda$. Then by straightforward calculations, (\ref{eqn:isotopy_notime}) is equivalent to the following
\begin{align} \label{eqn:isotopy_notime_invt}
	\dot{\lambda}_t &= Hi_v d\lambda_t+\overline{d_{\F} H}+v(H)(\lambda-\lambda_t) \\
					&= dH+Hi_v d\lambda_t-v(H)\lambda_t \nonumber\\
					&= i_v (Hd\lambda_t-dH\wedge\lambda_t) \nonumber\\
					&= \overline{d_{\F_t} H}-H\overline{\delta}_t, \nonumber
\end{align}
where $d\lambda_t=\overline{\delta}_t \wedge \lambda_t$ for some (uniquely determined) $\delta_t \in \Omega^1(\F_t)$. The existence of such $\delta_t$ is, of course, given by our assumption that $\F_t$ is a foliation for any $t$.

To summarize the above discussions, we have proved the following characterization of isotopies of foliations.

\begin{theorem} \label{thm:0order}
Suppose $\F_t$ is a $C^0$-small homotopy of foliations and $L$ is a fixed transverse line field. Let $\lambda_t=\lambda+\beta_t, \beta_0=0,$ be a path of defining one-forms such that $\beta_t(L)=0$ for all $t$. Then $\F_t$ is an isotopy of foliations if $[\dot{\lambda}_t]=[\dot{\beta}_t]=0 \in H^1_{tw} (\F_t)$ for all $t$. 
\end{theorem}

\begin{remark}
Although by \autoref{lem:well-defined-tangential}, the isomorphism class of $H^1_{tw}(\F_t)$ is independent of the choice of the defining one-form $\F=\ker\lambda_t$, we have to pick the presentative using $\lambda_t=\lambda+\beta_t$ in the above theorem.
\end{remark}

\subsubsection{Higher-order approximations}
Let us first consider the first-order approximation. Namely, consider a (time-dependent) tangential vector field $w \in \Gamma(T\F)$. Then $w$ naturally gives rise to a function $H_w$ on $T^\ast\F$ by the canonical pairing between $T\F$ and $T^\ast\F$, which is linear in the cotangent directions. In local coordinates as before, let us write $w=w^i \p_{q^i}$. Then we have $H_w=p_i w^i$. In this case, (\ref{eqn:isotopy_condition}) takes the following form
\begin{equation} \label{eqn:1orderP}
	\dot{p}_i = -\frac{\p p_i}{\p q^k} w^k-p_k \frac{\p w^k}{\p q^i}-\frac{1}{f} (p_k \frac{\p w^k}{\p x}+R^k p_j \frac{\p w^j}{\p q^k}-(\frac{\p f}{\p q^j}+p_k \frac{\p R^k}{\p q^j})w^j) p_i.
\end{equation}

Again by plugging (\ref{eqn:1orderP}) into (\ref{eqn:lambda_t}), we have
\begin{align} \label{eqn:isotopy_1order}
	\dot{\lambda} &= (R^k \dot{p}_k) dx-\dot{p}_i dq^i \\
				  &= \frac{1}{f} (p_k \frac{\p w^k}{\p x}+R^k p_j \frac{\p w^j}{\p q^k}-(\frac{\p f}{\p q^j}+p_k \frac{\p R^k}{\p q^j})w^j) (p_idq^i-(f+R^kp_k)dx) \nonumber \\
				  &\hspace{5mm} +(p_i \frac{\p w^i}{\p x}-\frac{\p (f+R^kp_k)}{\p q^i} w^i)dx+(\frac{\p p_i}{\p q^k} w^k+p_k \frac{\p w^k}{\p q^i}) dq^i \nonumber
\end{align}

After some calculations, one realizes that (\ref{eqn:isotopy_1order}) is equivalent to the following invariant form.
\begin{align*}
	\dot{\lambda}_t &= \lambda_t ([v,w]) \lambda_t-\mathcal{L}_w \lambda_t \\
				&= i_v (\mathcal{L}_w \lambda_t \wedge \lambda_t) \nonumber\\
				&= \lambda_t(w)\overline{\delta}_t-\overline{d_{\F_t} (\lambda_t(w))} \nonumber
\end{align*}

We observe that the first-order approximation does not give us anything new besides the result we obtained by looking at the 0$^{th}$-order approximation. This is because we can simply take the contact Hamiltonian function to be $H=\lambda_t(w)$. In fact, as we will see right below, the higher-order approximations do not any more information either.

Consider a multi-tangential-vector field $w_1 \otimes \cdots \otimes w_r \in \otimes^r T\F$, which in turns specifies a function $H \in C^\infty(T^\ast\F)$, homogeneous of degree $r$ in the cotangent directions. In local coordinates, we can write $H=w_1^{i_1} \cdots w_r^{i_r} p_{i_1} \cdots p_{i_r}$ where $w_j=w_j^k \p_{q^k}$ for any $j \in \{1,\cdots,r\}$. Using the notations from above, it is easily seen that $H=H_{w_1} \cdots H_{w_r}$. Playing the same game as before gives us the time-$t$ evolution of $\lambda_t$ as follows.
\begin{equation*}
	\dot{\lambda}_t = (\prod_{i=1}^r \lambda_t(w_i)) \overline{\delta}_t - \overline{d_{\F_t} (\prod_{i=1}^r \lambda_t(w_i))}.
\end{equation*}

So it is reasonable to form the following conjecture.
\begin{conj}
Suppose $\F_t$ is a $C^0$-small homotopy of foliations and $L$ is a fixed transverse line field. Let $\lambda_t=\lambda+\beta_t, \beta_0=0,$ be a path of defining one-forms such that $\beta_t(L)=0$ for all $t$. Then $\F_t$ is an isotopy of foliations if and only if $[\dot{\lambda}_t]=[\dot{\beta}_t]=0 \in H^1_{tw} (\F_t)$ for all $t$. 
\end{conj}

\section{Deformation of coisotropic submanifolds of arbitrary dimension} \label{sec:generalcoiso}

In this section, we consider the deformation theory of a $(n+1+k)$-dimensional coisotropic submanifold $Y^{n+1+k} \subset (M^{2n+1},\xi)$ for any $k \geq 0$. Setting $k=0$, we will partially recover the results for Legendrian foliations. In fact, as we will see, the deformation picture is considerably more involved when $k \geq 1$. Choose a contact form $\alpha$ and let $\lambda = \alpha|_Y \in \Omega^1(Y)$. Recall from \cite[Section 2.3]{H2013} that the {\em characteristic foliation} $\F = \ker(\lambda \wedge d\lambda)$ is a codimension $(2k+1)$ foliation on $Y$. As before, we will assume for the rest of this section that $\F$ is nonsingular.

\subsection{The contact master equation for coisotropic deformations} \label{subsec:MasterCoiso}
Given a coisotropic submanifold $Y \subset (M,\xi)$ with characteristic foliation $\F$ as above. Fix a transverse $(2k+1)$-dimensional distribution $G$ such that
\begin{equation} \label{eqn:split}
	TY=G \oplus T\F.
\end{equation}

To carry out explicit calculations, let us choose a foliated chart $\U \cong \RR^{n+1+k}_{\mathbf{y},\mathbf{q}}$ in $Y$ such that the leaves of $\F$ are defined by $\{\mathbf{q}=\text{const}\}$. Here we use the shorthand notation $\mathbf{y}=(y^1,\cdots,y^{2k+1})$ and $\mathbf{q}=(q^1,\cdots,q^{n-k})$. Locally we can write\footnote{Since we will be running out of Roman letters, we decide to use Greek letters as upper and lower indices for the rest of the paper. Of course the index $\alpha$ has nothing to do with the contact form $\alpha$.} $$G=\vspan \big\{ \frac{\p}{\p y^i}+R^\alpha_i \frac{\p}{\p q^\alpha} \big\}_{1 \leq i \leq 2k+1},$$ and $\lambda=a_i dy^i$ where $a_i \in C^\infty(\U)$ are smooth functions.

Just as in the case of nonsingular Legendrian foliations, we can write down a canonical contact model in a tubular neighborhood of $Y \subset (M,\xi)$ as follows. Consider the vector bundle
\begin{center}
$\begin{CD}
T^\ast\F @>>> E  @>\pi>> Y,
\end{CD}$
\end{center}
where the fibers are tangential cotangent spaces. Let $\Pi: TY \to T\F$ be the projection associated to the splitting (\ref{eqn:split}), and let $\mathbf{p}=(p_1,\cdots,p_{n-k})$ be the coordinates on the fiber $T^\ast\F$ dual to $\mathbf{q}$. Then we have the standard contact form $\alpha=\pi^\ast \lambda-\eta$ on $E$, where $\eta$ is defined  for any $v \in T_{(y,e)} E$, $y \in Y$ and $e \in T^\ast\F$, by
\begin{equation*}
	\eta(v)= \langle e,\Pi(\pi_\ast(v)) \rangle,
\end{equation*}
where $\langle~,~\rangle$ denotes the pairing between $T^\ast\F$ and $T\F$. In local coordinates we have
\begin{equation*}
	\eta=p_\alpha dq^\alpha-(p_\alpha R^\alpha_i) dy^i.
\end{equation*}
For simplicity of notations we will not distinguish $\lambda$ and $\pi^\ast\lambda$ hereafter. So we have
\begin{equation*}
	\alpha=\lambda+(p_\alpha R^\alpha_i) dy^i-p_\alpha dq^\alpha=(a_i+p_\alpha R^\alpha_i) dy^i-p_\alpha dq^\alpha.
\end{equation*}
One can easily show that the above defined $\alpha$ is contact in a sufficiently small neighborhood of the zero section. Moreover this contact form is canonical in the sense that it satisfies a Weistein-type neighborhood theorem similar to \cite[Lemma 3.2]{H2013}. The details are left to the interested readers.

Let $\mathbf{a}=(a_1,\cdots,a_{2k+1}) \in \RR^{2k+1}$ be the nonzero vector corresponding to the defining one-form $\lambda$, and consider the $2k$-dimensional subspace $\mathbf{a}^\bot=\{\mathbf{b} \in \RR^{2k+1}~|~\mathbf{a} \cdot \mathbf{b}=0\}$. Fix a basis $\mathbf{a}^\bot=\vspan \{\mathbf{b}_1,\cdots,\mathbf{b}_{2k}\}$, where $\mathbf{b}_i=(b_{i1},\cdots,b_{i,2k+1})$ for any $1 \leq i \leq 2k$. For $|\mathbf{p}|$ small enough, we notice that the following $(2k+1) \times (2k+1)$-matrix 
\begin{equation*}
\Phi= (\Phi_{ij})=
\begin{bmatrix}
	a_1+p_\alpha R^\alpha_1 & a_2+p_\alpha R^\alpha_2 & \ldots & a_{2k+1}+p_\alpha R^\alpha_{2k+1}  \\
	b_{11} & b_{12} & \ldots & b_{1,2k+1} \\
	\vdots & \vdots & \ddots & \vdots \\
	b_{2k,1} & b_{2k,2} & \ldots & b_{2k,2k+1}
\end{bmatrix}
\end{equation*}
is nonsingular, and we will denote $\Phi^{-1}=(\Phi^{ij})$ the inverse matrix. It is straightforward to check that
\begin{align*}
	\ker\alpha= & \vspan \big\{ \frac{\p}{\p p_\alpha} \big\}_{1 \leq \alpha \leq n-k} \oplus \vspan \big\{ b_{ij} (\frac{\p}{\p y^j}+ R^\gamma_j \frac{\p}{\p q^\gamma}) \big\}_{1 \leq i \leq 2k} \\
			   & \oplus \vspan \big\{ \frac{\p}{\p q^\beta}+p_\beta \Phi^{i1} \frac{\p}{\p y^i} \big\}_{1 \leq \beta \leq n-k}.
\end{align*}

For simplicity of notations, let us write 
\begin{equation} \label{eqn:basevector}
\mathfrak{e}_\alpha = \frac{\p}{\p p_\alpha}, \quad \mathfrak{f}_\beta= \frac{\p}{\p q^\beta}, \quad \mathfrak{g}_i=b_{ij} (\frac{\p}{\p y^j}+ R^\gamma_j \frac{\p}{\p q^\gamma})
\end{equation}
for the basis on $\ker\lambda$. 

Now let us take a detour to consider the ``transverse geometry'' of the characteristic foliation, which is analogous to the discussions in \cite{OP2005}. Namely, we define the transverse curvature on $G' := G \cap \ker\lambda=\vspan \{\mathfrak{g}_1,\cdots,\mathfrak{g}_{2k}\}$
\begin{equation*}
F_{G'}: G' \times G' \to \F
\end{equation*}
by $F_{G'} (v_1,v_2)=\Pi([v_1,v_2])$ for any $v_1,v_2 \in G'$, where $\Pi:TY \to T\F$ is the projection with respect to (\ref{eqn:split}). Using the above basis we have the following formula, which will be useful later.
\begin{align} \label{eqn:curvature}
F_G(\mathfrak{g}_i,\mathfrak{g}_j) &= \Pi \big( b_{is}(\frac{\p}{\p y^s}+R^\alpha_s \frac{\p}{\p q^\alpha}) (b_{jt}(\frac{\p}{\p y^t}+R^\beta_t \frac{\p}{\p q^\beta}))-b_{jt}(\frac{\p}{\p y^t}+R^\beta_t \frac{\p}{\p q^\beta}) (b_{is}(\frac{\p}{\p y^s}+R^\alpha_s \frac{\p}{\p q^\alpha})) \big) \nonumber\\
	&= b_{is} b_{jt} (\frac{\p R^\alpha_t}{\p y^s}+R^\beta_s \frac{\p R^\alpha_t}{\p q^\beta}-\frac{\p R^\alpha_s}{\p y^t}-R^\beta_t \frac{\p R^\alpha_s}{\p q^\beta}) \frac{\p}{\p q^\alpha} =: F^\alpha_{ij} \mathfrak{f}_\alpha
\end{align}

Then we can lift (\ref{eqn:basevector}) to a basis on $\ker \alpha$ as follows.
\begin{align*}
	 e_\alpha &= \Psi^{\gamma\alpha} \frac{\p}{\p p_\gamma}, \qquad \qquad  f^\beta = \frac{\p}{\p q^\beta}+p_\beta \Phi^{i1} \frac{\p}{\p y^i}+\Psi^{\eta\sigma} \frac{\p \Phi_{1j}}{\p q^\beta} p_\sigma \Phi^{j1} \frac{\p}{\p p_\eta}, \\
	 g_i &= b_{ij} \big( \frac{\p}{\p y^j}+ R^\gamma_j \frac{\p}{\p q^\gamma}+\Psi^{\eta \sigma}(p_\sigma \Phi^{s1} \frac{\p \Phi_{1s}}{\p y^j}-p_\sigma \Phi^{s1} \frac{\p \Phi_{1j}}{\p y^s}-\frac{\p \Phi_{1j}}{\p q^\sigma}+p_\sigma \Phi^{s1} R^\gamma_j \frac{\p \Phi_{1s}}{\p q^\gamma}) \frac{\p}{\p p_\eta} \big),
\end{align*}
where $\Psi^{-1}=(\Psi^{\alpha \beta})$ is the inverse matrix of
\begin{equation*}
\Psi= (\Psi_{\alpha \beta})_{(n-k)\times(n-k)}= I-(p_\alpha \Phi^{i1} R^\beta_i),
\end{equation*}
whose invertibility is ensured by assuming $|\mathbf{p}|$ is small.

The precise form of the lifting is quite messy, but the following lemma shows that they form a nice basis with respect to the symplectic form $d\alpha$ on $\ker \alpha$.

\begin{lemma}
The basis on $\ker\alpha=\vspan\{ e_1\cdots,e_{n-k},f^1,\cdots,f^{n-k},g_1,\cdots,g_{2k} \}$ satisfies the following properties
	\begin{itemize}
		\item $d\alpha(e_\alpha,e_\beta)=d\alpha(f^\alpha,f^\beta)=0$ for all $1 \leq \alpha,\beta \leq n-k$,
		\item $d\alpha(e_\alpha,f^\beta)=\delta_\alpha^\beta$ for all $1 \leq \alpha,\beta \leq n-k$, where $\delta$ is the Kronecker delta,
		\item $d\alpha(e_\alpha,g_i)=d\alpha(f^\alpha,g_i)=0$ for all $1 \leq \alpha \leq n-k$ and $1 \leq i \leq 2k$,
		\item $d\alpha(g_i,g_j)=d\lambda(\mathfrak{g}_i,\mathfrak{g}_j)+F^\alpha_{ij} p_\alpha$ for all $1 \leq i,j \leq 2k+1$, where $F^\alpha_{ij}$ is defined by (\ref{eqn:curvature}).
	\end{itemize}
\end{lemma}

\begin{proof}
We first note that
$$d\alpha=\frac{\p \Phi_{1j}}{\p y^i} dy^i \wedge dy^j+\frac{\p \Phi_{1i}}{\p q^\alpha} dq^\alpha \wedge dy^i+R^\alpha_i dp_\alpha \wedge dy^i-dp_\alpha \wedge dq^\alpha$$

The relation $d\alpha(e_\alpha,e_\beta) = 0$ is obvious, so we proceed to check that $d\alpha(f^\alpha,f^\beta) = 0$ as follows. For simplicity of notations, we temporarily write $A_{\alpha \beta}=\Psi^{\beta \gamma} p_\gamma \Phi^{i1} \frac{\p \Phi^{1i}}{\p q^\alpha}$. We have
\begin{align*}
d\alpha(f^\alpha,f^\beta) &= A_{\beta \alpha}-A_{\alpha \beta}+\frac{\p \Phi_{1i}}{\p q^\alpha} p_\beta \Phi^{i1}-\frac{\p \Phi_{1i}}{\p q^\beta} p_\alpha \Phi^{i1}-p_\alpha \Phi^{i1} R^\sigma_i A_{\beta \sigma}+p_\beta \Phi^{i1} R^\sigma_i A_{\alpha \sigma} \\
	&= A_{\beta \sigma} (\delta^\sigma_\alpha-p_\alpha \Phi^{i1} R^\sigma_i)-A_{\alpha \sigma}(\delta^\sigma_\beta-p_\beta \Phi^{i1} R^\sigma_i)+\frac{\p \Phi_{1i}}{\p q^\alpha} p_\beta \Phi^{i1}-\frac{\p \Phi_{1i}}{\p q^\beta} p_\alpha \Phi^{i1} \\
	&= A_{\beta \sigma} \Psi_{\alpha \sigma}- A_{\alpha \sigma} \Psi_{\beta \sigma}+ \frac{\p \Phi_{1i}}{\p q^\alpha} p_\beta \Phi^{i1}-\frac{\p \Phi_{1i}}{\p q^\beta} p_\alpha \Phi^{i1}=0
\end{align*}

Next we check
\begin{align*}
d\alpha(e_\alpha,f^\beta) &= \Psi^{\beta \gamma} p_\gamma \Phi^{i1} R^\alpha_i-\Psi^{\beta \alpha} \\
	&=\Psi^{\beta \gamma} \big( p_\gamma \Phi^{i1} R^\alpha_i-\delta_\gamma^\alpha \big)= -\delta_\alpha^\beta.
\end{align*}

We leave the details of checking the remaining part of the lemma to the interested reader.
\end{proof}

Let $s: Y \to E$ be a section, or equivalently, a tangential one-form in $\Omega^1(\F)$. We assume that $s$ is sufficiently close to the zero section where the contact structure is defined. In local coordinates as above, we can write the section as $p_\alpha=s_\alpha(\mathbf{y},\mathbf{q})$ for $1 \leq \alpha \leq n-k$. Let $Y_s$ be the graph of $s$ in $E$. Then its tangent space is spanned by
\begin{equation*}
TY_s=\big\{\frac{\p}{\p y^i}+\frac{\p s_\beta}{\p y^i} \frac{\p}{\p p_\beta} \big\}_{1 \leq i \leq 2k+1} \oplus \big\{ \frac{\p}{\p q^\alpha}+\frac{\p s_\gamma}{\p q^\alpha} \frac{\p}{\p p_\gamma} \big\}_{1 \leq \alpha \leq n-k}.
\end{equation*}
Then one can easily check that
\begin{align*}
TY_s \cap \xi &= \big\{ s_\alpha \Phi^{i1}(\frac{\p}{\p y^i}+\frac{\p s_\beta}{\p y^i} \frac{\p}{\p p_\beta})+\frac{\p}{\p q^\alpha}+\frac{\p s_\beta}{\p q^\alpha} \frac{\p}{\p p_\beta} \big\}_{1 \leq \alpha \leq n-k} \\
	&\quad \oplus \big\{ b_{ij}(\frac{\p}{\p y^j}+\frac{\p s_\beta}{\p y^j} \frac{\p}{\p p_\beta}+R^\gamma_j(\frac{\p}{\p q^\gamma}+\frac{\p s_\beta}{\p q^\gamma} \frac{\p}{\p p_\beta})) \big\}_{1 \leq i \leq 2k} \\
	&= \big\{ f^\alpha+\mathcal{A}_{\alpha \beta} e_\beta \big\}_{1 \leq \alpha \leq n-k} \oplus \big\{ g_i+\mathcal{B}_{i\alpha} e_\alpha \big\}_{1 \leq i \leq 2k},
\end{align*}
where
\begin{align}
	\mathcal{A}_{\alpha \beta} &= (s_\alpha \Phi^{i1} \frac{\p s_\gamma}{\p y^i}+\frac{\p s_\gamma}{\p q^\alpha}) \Psi_{\beta\gamma}-(\frac{\p a_i}{\p q^\alpha}+s_\gamma \frac{\p R^\gamma_i}{\p q^\alpha}) s_\beta \Phi^{i1}, \text{ and} \label{eqn:A's} \\
	\mathcal{B}_{i\alpha} &= b_{ij} \big( (\frac{\p s_\gamma}{\p y^j}+R^\nu_j \frac{\p s_\gamma}{\p q^\nu}) \Psi_{\alpha\gamma} + s_\alpha \Phi^{s1} (\frac{\p \Phi_{1j}}{\p y^s}-\frac{\p \Phi_{1s}}{\p y^j}) +\frac{\p \Phi_{1j}}{\p q^\alpha} - s_\alpha R^\nu_j \Phi^{s1} \frac{\p \Phi_{1s}}{\p q^\nu} \big) \label{eqn:B's}
\end{align}

With all the preparations above, the following characterization of coisotropic deformation now follows directly from \cite[Proposition 2.2]{OP2005}. 

\begin{prop} \label{prop:coiso_master}
The graph $Y_s$ of a section $s: Y \to E$, viewed as a one-form in $\Omega^1(\F)$, is coisotropic if and only if the following equation holds.
\begin{equation} \label{eqn:coiso_master}
	\mathcal{A}_{\beta\alpha}-\mathcal{A}_{\alpha\beta} = \mathcal{B}_{i\alpha} \omega^{ij} \mathcal{B}_{j\beta},
\end{equation}
where $(\omega^{ij})$ is the inverse matrix to the matrix $(d\lambda(\mathfrak{g}_i,\mathfrak{g}_j)+F^\alpha_{ij}s_\alpha)$.
\end{prop}

We will call (\ref{eqn:coiso_master}) the contact master equation of coisotropic deformations. This generalizes the case of Legendrian foliations computed in \autoref{thm:master}. For later purposes, let us explicitly compute the LHS of (\ref{eqn:coiso_master}) as follows.
\begin{equation} \label{eqn:LHS}
	\mathcal{A}_{\alpha\beta}-\mathcal{A}_{\beta\alpha}=\frac{\p s_\beta}{\p q^\alpha}-\frac{\p s_\alpha}{\p q^\beta}+\Phi^{i1} s_\alpha(\frac{\p \Phi_{1i}}{\p q^\beta}+\frac{\p s_\beta}{\p y^i})-\Phi^{i1} s_\beta(\frac{\p \Phi_{1i}}{\p q^\alpha}+\frac{\p s_\alpha}{\p y^i})
\end{equation}

\subsection{Infinitesimal deformation of coisotropic submanifolds} \label{sec:coiso_infi}
Just as in the Legendrian foliation case, we first investigate the linearized version of (\ref{eqn:coiso_master}). This is easily seen, using (\ref{eqn:LHS}), to be
\begin{equation}
\frac{\p s_\beta}{\p q^\alpha}-\frac{\p s_\alpha}{\p q^\beta}+\sum_{i=1}^{2k+1} \frac{a_i}{|a|^2} (s_\alpha \frac{\p a_i}{\p q^\beta}-s_\beta \frac{\p a_i}{\p q^\alpha})=0,
\end{equation}
for any $1 \leq \alpha,\beta \leq n-k$, where $|a|^2=\sum_{i=1}^{2k+1} a^2_i$. Here we used the fact that the RHS of (\ref{eqn:coiso_master}) is of order at least two, and also the following Taylor expansion
\begin{equation*}
	\Phi^{i1} = \frac{a_i}{\sum_{j=1}^{2k+1} a_j^2+s_\alpha R^\alpha_j a_j} = \frac{a_i}{|a|^2}+o(s).
\end{equation*}

Let $\lambda$ be the restriction of the contact form $\alpha$ on the coisotropic submanifold $Y$ as before, and $\F=\ker(\lambda \wedge d\lambda)$ be the characteristic foliation. We will define a tangential one-form associated with $\lambda$, which generalizes our construction of $\mu$ (or rather $\pi(\mu)$) in \autoref{sec:leg_infi}. Namely, for any tangential vector field $X \in \Gamma(\F)$, observe that $i_X (\lambda \wedge d\lambda)=\lambda \wedge (i_X d\lambda)=0$. Hence there exists a function $\mu_X$ such that $i_X d\lambda=\mu_X \lambda$. Now we simply define the one-form $\mu \in \Omega^1(\F)$ by $\mu(X):=\mu_X$ for any $X \in \Gamma(\F)$.

\begin{lemma}
The above defined $\mu \in \Omega^1(\F)$ is uniquely determined by $\lambda$, and moreover it is $d_\F$-closed.
\end{lemma}

\begin{proof}
The uniqueness part is obvious. So we will prove the $d_\F$-closeness. In a foliated chart $\RR^{n+k+1}_{\mathbf{y},\mathbf{q}}$, we can write $\lambda=a_i dy^i$ and compute its exterior derivative $$d\lambda=\frac{\p a_j}{\p y^i} dy^i \wedge dy^j+\frac{\p a_i}{\p q^\alpha} dq^\alpha \wedge dy^i.$$ To compute $\mu=\mu_\alpha dq^\alpha$, we note that by definition $$(\mu_\alpha a_i) dy^i=\mu_\alpha \lambda=i_{\p_{q^\alpha}} d\lambda=\frac{\p a_i}{\p q^\alpha}dy^i,$$ which, by equating both sides term-by-term, is equivalent to
\begin{equation} \label{eqn:mu_alpha_pre}
\mu_\alpha a_i=\frac{\p a_i}{\p q^\alpha} \quad \text{for all } 1 \leq i \leq 2k+1.
\end{equation}
Now we can multiply both sides of (\ref{eqn:mu_alpha_pre}) by $a_i$ and sum over $i$ to obtain
\begin{equation} \label{eqn:mu_alpha}
\mu=\mu_\alpha dq^\alpha=\sum_{i=1}^{2k+1} \frac{a_i}{|a|^2} \frac{\p a_i}{\p q^\alpha} dq^\alpha=d_\F(\log |a|).
\end{equation}
The conclusion that $\mu$ is $d_\F$-closed follows from the fact that $d_\F$ is a differential.
\end{proof}

It is easy to see that everything in \autoref{sec:leg_infi} can be carried over here with little modifications, so we will omit the details. Note however that here we just define $\mu$, instead of $\pi(\mu)$, to be the tangential one-form. In particular we define the twisted tangential de Rham cohomology as follows.

\begin{definition}
The {\em twisted tangential differential} $d_\F^\lambda$ on $\Omega^\bullet(\F)$ is defined by $$d_\F^\lambda(\omega)=d_\F \omega-\mu\wedge\omega$$ for any $\omega \in \Omega^\bullet(\F)$. The {\em twisted tangential (de Rham) cohomology} $H_t^\bullet(\F)$ is defined to be $H^\bullet(\Omega^\bullet(\F),d_\F^\lambda)$.
\end{definition}

The following theorem characterizes the infinitesimal deformations of general coisotropic submanifolds with nonsingular characteristic foliation. This is completely analogous to \autoref{thm:leg_infi}.

\begin{theorem}
Let $Y \subset (M,\xi=\ker \alpha)$ be a coisotropic submanifold with nonsingular characteristic foliation $\F$, and let $\lambda=\alpha|_Y$. Then a tangential one-form $\zeta \in \Omega^1(\F)$ is an infinitesimal coisotropic deformation of $Y$ if and only if $d_\F^\lambda \zeta=0$ on $\Omega^2(\F)$.
\end{theorem}

\begin{proof}
The conclusion follows immediately from the following formula
\begin{equation*}
d_\F^\lambda \zeta= \frac{\p \zeta_\gamma}{\p q^\beta} dq^\beta \wedge dq^\gamma-\sum_{i=1}^{2k+1} \frac{a_i}{|a|^2} \frac{\p a_i}{\p q^\beta} \zeta_\gamma dq^\beta \wedge dq^\gamma
\end{equation*}
where $\zeta=\zeta_\beta dq^\beta \in \Omega^1(\F)$, which is easy to check using (\ref{eqn:mu_alpha}).
\end{proof}

\subsection{Pre-contact structures} \label{subsec:precontact}
It will turn out that the deformation of coisotropic submanifolds is closely related to the deformation of the so-called pre-contact structures. So we briefly review some basic definitions and facts about pre-contact structures in this section. 

Given a closed manifold $W$ of dimension $n$. A {\em pre-contact structure of rank $k$} on $W$ is a hyperplane distribution $\xi$ such that for any defining one-form $\xi=\ker\alpha$, the rank rk$(d\alpha|_{\ker\alpha}) \equiv k$. This definition is clearly motivated by the similar definition of pre-symplectic structures. It is easy to see that a pre-contact structure of rank 0 is a foliation, and if $n=2r+1$ is odd, then a pre-contact structure of full rank $2r$ is a contact structure. Also note that $k$ is always an even number since $d\alpha$ is a skew symmetric bilinear form on $\ker\alpha$.

Any pre-contact manifold $(W,\xi)$ of rank $k$ admits a characteristic foliation $\F=\ker(\alpha \wedge d\alpha)$ of codimension $k+1$. See, for example, \cite[Lemma 2.15]{H2013} for a proof of this fact. Just as in \autoref{subsec:MasterCoiso}, let us fix a transverse distribution $G$ such that $TW=T\F \oplus G$. Then we will say $\alpha$ defines a contact structure on $G$ in the sense that $\alpha \wedge (d\alpha)^{k/2}$ is nondegenerate on $G$. This is, of course, equivalent to saying that $d\alpha$ is nondegenerate on $\xi \cap G$.

We are mostly interested in deformations of pre-contact structures. So let $\xi_t$ be a smooth path of pre-contact structures of rank $k$ on $W$ with $\xi_0=\xi$. The special case when $k=0$, i.e., a path of foliations, tells us that one cannot hope to find an ambient isotopy $\phi_t: W \to W$ such that $\phi_t^\ast(\xi_t)=\xi$. In other words, pre-contact structure is not a stable structure unless it is a contact structure. However, since $\xi$ restricts to a contact structure on $G$, one can hope for a stability result in the $G$-directions. This is indeed the case by the following lemma.

\begin{lemma} \label{lem:precontact}
Given a $C^1$-small path of pre-contact structures $\xi_t, t \in [0,1]$, on $W$ of constant rank. Fix a splitting $TW=T\F \oplus G$, where $\F$ is the characteristic foliation of $\xi_0$. Then there exists an ambient isotopy $\phi_t: W \to W$ such that $\phi_0=\text{\em id}$ and $\phi_t^\ast \xi_t \cap G=\xi_0 \cap G$.
\end{lemma}

\begin{proof}
The proof is a standard application of Moser's trick. Namely, let $\xi_t=\ker\alpha_t$. We look for solutions to the following equation
\begin{equation} \label{eqn:moser1}
	\phi_t^\ast \alpha_t|_G=g\alpha_0|_G,
\end{equation}
for some $g \in C^\infty(W)$. Now differentiate both sides of (\ref{eqn:moser1}) with respect to $t$ and simplify to arrive at the following equation.
\begin{equation} \label{eqn:moser2}
	\mathcal{L}_{V_t} \alpha_t|_G+\dot\alpha_t|_G=h\alpha_t|_G
\end{equation}

Observe that for each $t$, since $\xi_t \cap G$ is contact by the $C^1$-small condition, there is a unique vector field $R_t \in \Gamma(G)$ such that $\alpha_t(R_t)=1$ and $i_{R_t} d\alpha_t=0$. Then we can write $V_t=f_t R_t+X_t$, where $\alpha_t(X_t)=0$. Then using Cartan's formula, we can rewrite (\ref{eqn:moser2}) as follows,
\begin{equation}
	i_{X_t} d\alpha_t|_G+(df_t+\dot{\alpha}_t)|_G=h\alpha_t|_G.
\end{equation}
For any choice of $f_t$, there is a (not necessarily unique) solution for $X_t$ and $h$. This finishes the proof.
\end{proof}

For later use, let us make the following definition.

\begin{definition} \label{defn:precontact}
A path $\xi_t, t \in [0,1]$, of pre-contact structures of constant rank is called {\em strict} if for some choice of the transverse distribution $G$, the contact structure on $\xi_t \cap G$ is constant in time $t$. In this case, we say $\xi_1$ is {\em strictly homotopic} to $\xi_0$ with respect to $G$.
\end{definition}

The following result follows immediately from \autoref{lem:precontact}.

\begin{cor} \label{cor:precontact}
For any pre-contact structure $\xi$ on $W$, if another pre-contact structure $\xi'$ is connected to $\xi$ by a $C^1$-small path of pre-contact structures of constant rank, then there exists a pre-contact structure $\xi''$, which is isotopic to $\xi'$ and is strictly homotopic to $\xi$.
\end{cor}

\begin{proof}
Since $\xi$ and $\xi'$ are connected by a $C^1$-small path of pre-contact structures, \autoref{lem:precontact} produces an isotopy $\phi_t: W \to W, t \in [0,1]$. It is easy to see $\xi''=\phi_1(\xi')$ satisfies all the desired conditions.
\end{proof}

\subsection{Deformation of coisotropic submanifolds and pre-contact structures}
In this section we will try to generalize the results in \autoref{subsec:foliation} and \autoref{subsec:gauge_equiv} to the case of coisotropic submanifolds $Y^{n+1+k} \subset (M^{2n+1},\xi)$ with nonsingular characteristic foliation. Obviously in this situation, the pre-contact structures discussed in the previous section will play the role of foliations. However, due to the nontrivial ``transverse geometry'' of the characteristic foliation for $k>0$, we cannot recover the full strength of \autoref{thm:leg_deform} and \autoref{thm:ModGauge}. Nevertheless we will build up a partial correspondence between the deformation of coisotropic submanifolds and the deformation of pre-contact structures.

Suppose $(Y,\F) \subset (M,\xi)$ is a coisotropic submanifold with nonsingular characteristic foliation $\F$. Fix a transverse distribution $G \pitchfork \F$ as before. Let $s \in \Omega^1(\F)$ be a tangential one-form. Define $\overline{s} \in \Omega^1(Y)$ to be the lift of $s$ defined by asking $\overline{s}(X)=0$ for any vector field $X \in \Gamma(G)$. In local foliated coordinates as before, if we write $s=s_\alpha dq^\alpha$, then one can easily check $\overline{s}=s_\alpha dq^\alpha-(R^\beta_i s_\beta) dy^i$. Motivated by the discussions in \autoref{subsec:foliation}, we are interested in the following one-form
\begin{equation*}
	\lambda-\overline{s}=\Phi_{1i} dy^i-s_\alpha dq^\alpha \in\Omega^1(Y), 
\end{equation*}
where $\Phi_{1i}=a_i+R^\gamma_i s_\gamma$. It is easy to see that
\begin{equation*}
	\ker(\lambda-\overline{s})=\vspan \big\{ \frac{\p}{\p q^\alpha}+\Phi^{i1}s_\alpha \frac{\p}{\p y^i} \big\}_{1 \leq \alpha \leq n-k} \oplus \vspan \big\{ b_{ij} (\frac{\p}{\p y^j}+R^\gamma_j \frac{\p}{\p q^\gamma}) \big\}_{1 \leq i \leq 2k}.
\end{equation*}

For the convenience of notations, let us write
\begin{equation*}
	\mathfrak{F}_\alpha = \frac{\p}{\p q^\alpha}+\Phi^{i1}s_\alpha \frac{\p}{\p y^i} \quad \text{and} \quad \mathfrak{G}_i = b_{ij} (\frac{\p}{\p y^j}+R^\gamma_j \frac{\p}{\p q^\gamma}).
\end{equation*}
Here one may notice that $\mathfrak{G}_i = \mathfrak{g}_i$ defined in (\ref{eqn:basevector}), which is so because we constructed $\overline{s}$ to be vanishing on the $G$-directions. Nevertheless we decided to use a different notation here to avoid any possible confusions.

The following calculations are straightforward but somewhat tedious, so we will just give the results.
\begin{align*}
	d(\lambda-\overline{s})(\mathfrak{F}_\alpha,\mathfrak{F}_\beta) &= \frac{\p s_\alpha}{\p q^\beta}-\frac{\p s_\beta}{\p q^\alpha}+\Phi^{t1} s_\beta(\frac{\p \Phi_{1t}}{\p q^\alpha}+\frac{\p s_\alpha}{\p y^t})-\Phi^{t1} s_\alpha(\frac{\p \Phi_{1t}}{\p q^\beta}+\frac{\p s_\beta}{\p y^t}), \\
	d(\lambda-\overline{s})(\mathfrak{F}_\alpha,\mathfrak{G}_i) &= b_{ij} \big( (\frac{\p s_\alpha}{\p q^\gamma}-\frac{\p s_\gamma}{\p q^\alpha}) R^\gamma_j-\Phi^{t1} s_\alpha R^\gamma_j (\frac{\p s_\gamma}{\p y^t}+\frac{\p \Phi_{1t}}{\p q^\gamma})+\frac{\p s_\alpha}{\p y^j}+\frac{\p \Phi_{1j}}{\p q^\alpha} \\
	& \quad -\Phi^{t1}s_\alpha (\frac{\p \Phi_{1t}}{\p y^j}-\frac{\p \Phi_{1j}}{\p y^t}) \big) \nonumber
\end{align*}

After some calculations, one can show that
\begin{equation*}
	\Omega_{\alpha\beta}:=d(\lambda-\overline{s}) (\mathfrak{F}_\alpha,\mathfrak{F}_\beta) = \mathcal{A}_{\beta\alpha}-\mathcal{A}_{\alpha\beta} \quad \text{and} \quad \Omega_{\alpha i}:=d(\lambda-\overline{s})(\mathfrak{F}_\alpha,\mathfrak{G}_i) = \mathcal{B}_{i\alpha},
\end{equation*}
where $\mathcal{A}_{\alpha\beta}$ and $\mathcal{B}_{i\alpha}$ are defined by (\ref{eqn:A's}) and (\ref{eqn:B's}), respectively. Moreover we define $\Omega_{ij}:= d(\lambda-\overline{s})(\mathfrak{G}_i,\mathfrak{G}_j)$, and denote $(\Omega^{ij})_{2k \times 2k}$ the inverse matrix of $(\Omega_{ij})_{2k \times 2k}$. Here the invertibility is guaranteed by assuming $s$ to be sufficiently small.

Now we check the condition under which the rank of $d(\lambda-\overline{s})$, or equivalently the rank of the above defined $(n-k) \times (n+k)$ matrix $\Omega$, is constantly equal to $2k$ on $\ker(\lambda-\overline{s})$. Since $(\Omega_{ij})_{2k \times 2k}$ is invertible and has rank $2k$, the condition is equivalent to
\begin{equation} \label{eqn:Omega}
	\Omega(\mathfrak{F}_\alpha-\Omega_\alpha^i \mathfrak{G}_i, \mathfrak{F}_\beta-\Omega_\beta^j \mathfrak{G}_j) =0
\end{equation}
for all $1 \leq \alpha,\beta \leq n-k$, where $\Omega_\alpha^i := \Omega_{\alpha l} \Omega^{li}$ is defined by the usual lifting of indices. Here, of course, we choose $\mathfrak{F}_\alpha-\Omega_\alpha^i \mathfrak{G}_i$ such that $d(\lambda-\overline{s})(\mathfrak{F}_\alpha-\Omega_\alpha^i \mathfrak{G}_i,\mathfrak{G}_j)=0$ for any $1 \leq j \leq 2k$. Spanning out the terms in (\ref{eqn:Omega}) and a little simplification yields the following equivalent equation
\begin{equation}
 \Omega_{\alpha\beta}= \Omega_{\alpha i} \Omega^{ji} \Omega_{\beta j}.
\end{equation}

Finally we observe that $\Omega_{ij}=-\omega_{ij}$, which is defined in \autoref{prop:coiso_master}. This is so by a straightforward computation as follows.
\begin{align*}
d\overline{s}(\mathfrak{G}_i,\mathfrak{G}_j) &= \big( \frac{\p s_\beta}{\p q^\alpha} dq^\alpha \wedge dq^\beta+(\frac{\p s_\alpha}{\p y^i}+\frac{\p (R_i^\beta s_\beta)}{\p q^\beta}) dy^i \wedge dq^\alpha+\frac{\p (R^\beta_i s_\beta)}{\p y^j} dy^i \wedge dy^j \big) (\mathfrak{G}_i,\mathfrak{G}_j) \\
	&= b_{is}b_{jt} \big( \frac{\p (R^\gamma_s s_\gamma)}{\p y^t}-\frac{\p (R^\gamma_t s_\gamma)}{\p y^s}+R^\gamma_t (\frac{\p s_\gamma}{\p y^s}+\frac{\p (R^\beta_s s_\beta)}{\p q^\gamma}) - R^\gamma_s (\frac{\p s_\gamma}{\p y^t}+\frac{\p (R^\beta_t s_\beta)}{\p q^\gamma}) \\
	& \quad+ R_s^\alpha R_t^\beta (\frac{\p s_\beta}{\p q^\alpha}-\frac{\p s_\alpha}{\p q^\beta}) \big) \\
	&= -F^\gamma_{ij} s_\gamma
\end{align*}
where $F^\gamma_{ij}$ is the curvature coefficient defined by (\ref{eqn:curvature}).

To summarize, we have proved the following result, which partially generalizes \autoref{thm:leg_deform} in the case of Legendrian foliations.

\begin{theorem}
Suppose $(Y,\F) \subset (M,\xi)$ is a coisotropic submanifold with nonsingular characteristic foliation $\F$. Let $\xi|_Y$ be the restricted pre-contact structure on $Y$. Fix a transverse distribution $G$ such that $TY=G \oplus T\F$. Then for any coisotropic submanifold $Y'$ which is $C^1$-close to $Y$, we can canonically associate a pre-contact structure $\xi'$ on $Y$ corresponding to $Y'$, which is $C^1$-close to $\xi|_Y$. Moreover $Y'$ is homotopic to $Y$ through a $C^1$-small path of coisotropic submanifolds if and only if $\xi'$, up to an isotopy, is strictly homotopic to $\xi|_Y$ through a $C^1$-small path of pre-contact structures of constant rank.
\end{theorem}

\begin{proof}
Choose a contact form $\xi=\ker\alpha$, and let $\lambda=\alpha|_Y$ be the corresponding pre-contact form on $Y$ as usual. The first statement follows directly from the above discussion. Namely, given a coisotropic submanifold $Y'$ which is $C^1$-close to $Y$, we can realize $Y'$ as the graph of a tangential one-form $s \in \Omega^1(\F)$ in a standard contact neighborhood. Then the above calculations show that $\xi':=\ker(\lambda-\overline{s})$ is a pre-contact structure, which is $C^1$-close to $\xi|_Y$ given $s$ is small.

The second statement is now obvious. Namely, given a graphical path of coisotropic submanifolds in a neighborhood of $Y$, we obtain a corresponding path of tangential one-forms which in turn produces a path of pre-contact structures on $Y$ in the obvious way. Conversely, suppose we are given a $C^1$-small homotopy of pre-contact structures $\mu_t, t \in [0,1]$ on $Y$, such that $\mu_0=\xi|_Y$ and $\mu_1=\xi'$. Then by \autoref{cor:precontact}, we may assume, up to an isotopy, that $\xi'$ is strictly homotopic to $\xi$ in the sense of \autoref{defn:precontact}. Then multiplying the homotopy $\xi_t$ by a (time-dependent) non-vanishing function if necessary, we get a a sequence of defining one-form $\xi_t=\ker\lambda_t$ of the forms $\lambda_t=\lambda-\overline{s}_t$. The graph of $s_t$ in a standard contact neighborhood of $Y$ gives the desired homotopy of coisotropic manifolds from $Y'$ to $Y$.
\end{proof}

\appendix
\section{Twisted tangential de Rham cohomology} \label{apx:twisted_de_rham}

In this appendix, we will show that given a codimension one foliation, the twisted tangential (de Rham) cohomology is in general different from the usual tangential cohomology by an explicit computation the following example.

\begin{example}
Consider the one-dimensional foliation $\F$ on $T^2_{x,y}=\mathbb{R}^2 / (2\pi\mathbb{Z})^2$ defined by $\F=\ker\lambda$, where $\lambda=dy+\sin ydx$. Note that $\F$ has two closed leaves $\{y=0\}$ and $\{y=\pi\}$, and all the other leaves are diffeomorphic to $\mathbb{R}$, and they all limit onto the two closed leaves.
	
Let us first look at the usual tangential complex, which is given by
\begin{equation}
	0 \to \Omega^0(\F) \xrightarrow{d_\F} \Omega^1(\F) \to 0.
\end{equation}
We have of course $\Omega^0(\F)=C^\infty(T^2)$. By choosing a transverse line field $L=\langle \p_y \rangle$, we can identify $\Omega^1(\F)$ with the set $\{ gdx ~|~ g \in C^\infty(T^2) \}$. Moreover, for any $f \in \Omega^0(T^2)$, we have
\begin{equation}
	d_\F(f)=(\p_x f-\sin y \p_y f)dx
\end{equation}
Due to the dynamics of $\F$, it is easy to see that
\begin{equation}
	H^0(\F) \cong \mathbb{R},
\end{equation}	
which is generated by the constant functions. 

To compute $H^1(\F)$, let us pick any $fdx \in \Omega^1(\F)$. Let us call $a_0=\frac{1}{2\pi}\int_{S^1} f(x,0)dx$ and $a_\pi=\frac{1}{2\pi}\int_{S^1} f(x,\pi)dx$ the {\em periods} of $f$. Then it is clear that up to a coboundary, we can assume $f(x,0) \equiv a_0$ and $f(x,\pi) \equiv a_\pi$. Moreover two tangential one-forms represent different de Rham classes if they have different periods. Now suppose $a_0=a_\pi=0$. Then by restricting $f$ to the leaves contained in the cylinder $\{0<y<\pi\}$, we obtain a $S^1$-family of Schwartz functions on $\mathbb{R}$
\begin{equation}
	\phi_{(0,\pi)}: S^1 \to \mathcal{S}(\mathbb{R}).
\end{equation}
Similarly we can also restrict $f$ to the other cylinder $\{\pi<y<2\pi\}$ to get a map
\begin{equation}
	\phi_{(\pi,2\pi)}: S^1 \to \mathcal{S}(\mathbb{R}). 
\end{equation}
	
Conversely any two maps $\phi_{(0,\pi)}$ and $\phi_{(\pi,2\pi)}$ as above determines a one-cocycle in $\Omega^1(\F)$, which is a coboundary if and only if the corresponding functions
\begin{equation}
 	\Phi_{(0,\pi)}: S^1 \to \mathbb{R} \hspace{5mm} \text{and} \hspace{5mm} \Phi_{(\pi,2\pi)}: S^1 \to \mathbb{R}
\end{equation} 
obtained by integrating along $\mathbb{R}$ are zero functions.

As a conclusion, we have computed
\begin{equation}
	H^1(\F)=\mathbb{R}^2 \oplus (\text{Map}(S^1,\mathbb{R}))^{\oplus 2}
\end{equation}
as an infinite dimensional vector space.

Now let us look at the twisted version. First note that given the choice of $L=\langle \p_y \rangle$ as above, we have $\mu=-\cos ydx \in \Omega^1(\F)$, where $\mu$ is such that $d\lambda=\mu \wedge \lambda$. So $g \in H^0_{tw}(\F)$ if and only it satisfies the following first-order inhomogeneous PDE
\begin{equation} \label{eqn:PDE1}
	\p_x g-\sin y \p_y g+g\cos y=0
\end{equation}
on $T^2$. It turns out that the only solution to (\ref{eqn:PDE1}) is $\{g \equiv 0\}$. So we have $H^0_{tw}(\F)=0$.

Next we will compute $H^1_{tw}(\F)=\Omega^1(\F) / \image(d_{\F}^\lambda)$. Let us first restrict to the circle $\{y=0\}$, then $h \in \image(d_\F^\lambda)$ if
\begin{equation} \label{eqn:ODE1}
	h(x,0)=\p_x f(x,0)+f(x,0)
\end{equation}
for some $f \in C^\infty(T^2)$, where $x \in \mathbb{R}/2\pi\mathbb{Z}$. For simplicity of notations, let us write $f(x)=f(x,0)$ and $h(x)=h(x,0)$ for the moment. Now we can explicitly solve (\ref{eqn:ODE1}) to get
\begin{equation}
	f(x)=e^{-x} (\int_0^x h(t)e^t dt+C),
\end{equation}
where $C$ is a constant. Since $f$ is periodic of period $2\pi$, we can solve for $C$ and get the following (periodic) solution to (\ref{eqn:ODE1}).
\begin{equation}
	f(x)=e^{-x}  (\int_0^x h(t)e^t dt+\frac{1}{e^{2\pi}-1} \int_0^{2\pi} h(t)e^t dt)
\end{equation}
Hence we find out that $d_\F^\lambda$ is surjective when restricted to $\{y=0\}$. Similar calculation shows that the same holds for the other closed leaf $\{y=\pi\}$. Therefore to compute $H^1_{tw}(\F)$ we may restrict ourself to one-forms $hdx \in \Omega^1(\F)$ such that $h \equiv 0$ on $\{y=0\} \cup \{y=\pi\}$.

As in the untwisted case, let us consider an open leaf $F$ contained in the cylinder $\{0<y<\pi\}$. We can parametrize $F$ by a diffeomorphism $\phi: \mathbb{R} \to F$ defined by $\phi(t)=(t,2\cot^{-1}(e^t))$. Then clearly $d^\lambda_\F |_F: \mathcal{S}(\mathbb{R}) \to \mathcal{S}(\mathbb{R})$ is defined by $d^\lambda_\F |_F (f)=\p_t f+f\tanh t$. Now the equation
\begin{equation}
	h(t)=\p_t f(t)+f(t)\tanh t
\end{equation}
can be solved explicitly by
\begin{equation}
	f(t)=\sech(t) \int_0^t h(s)\cosh sds.
\end{equation}
It is not hard to show that $f \in \mathcal{S}(\mathbb{R})$ if $h$ is. Similarly argument also works for any leaf contained in $\{0<y<\pi\}$. To summarize, we have proved that the twisted tangential complex $(\Omega^\bullet(\F),d_\F^\lambda)$ is acyclic, i.e., $H^0_{tw}(\F)=H^1_{tw}(\F)=0$.
\end{example}

The following observation follows immediately from \autoref{cor:infini}.

\begin{cor}
Let $\F=\ker(dy+\sin y dx)$ be the foliation on $T^2$ as above. Then any foliation, which is $C^1$-close to $\F$, is ambient isotopic to $\F$.
\end{cor}

The above example shows that $H^\bullet_{tw}(\F)$ is general very different from $H^\bullet(\F)$, but they do coincide if $\F$ is defined by a closed one-form. The general properties of the twisted tangential cohomology are largely unknown at this moment.

%\begin{qn}
%	What can one say about the twisted tangential cohomology $H^\bullet_{tw}(\F)$ in general, or rather, what information of $\F$ does $H^\bullet_{tw}(\F)$ carry, say, some kind of stability?
%\end{qn}

\bibliography{mybib}

\providecommand{\bysame}{\leavevmode\hbox to3em{\hrulefill}\thinspace}
\providecommand{\MR}{\relax\ifhmode\unskip\space\fi MR }
\providecommand{\MRhref}[2]{%
  \href{http://www.ams.org/mathscinet-getitem?mr=#1}{#2}
}
\providecommand{\href}[2]{#2}
\begin{thebibliography}{EKAN93}

\bibitem[EKAN93]{EKN1993}
Aziz El~Kacimi-Alaoui and Marcel Nicolau, \emph{A class of {$C^\infty$}-stable
  foliations}, Ergodic Theory Dynam. Systems \textbf{13} (1993), no.~4,
  697--704. \MR{1257030 (94m:58173)}

\bibitem[Ham78]{Ham1978}
Richard~S. Hamilton, \emph{Deformation theory of foliations}, preprint, 1978.

\bibitem[Hei75]{He1975}
James~L. Heitsch, \emph{A cohomology for foliated manifolds}, Comment. Math.
  Helv. \textbf{50} (1975), 197--218. \MR{0372877 (51 \#9081)}

\bibitem[Hua13]{H2013}
Yang Huang, \emph{On {L}egendrian foliations in contact manifolds {I}:
  Singularities and neighborhood theorems}, preprint, 2013.

\bibitem[Hua14]{H2014}
\bysame, \emph{Non-existence of certain singularities in {L}egendrian
  foliations}, preprint, 2014.

\bibitem[LOTV14]{LOTV}
H\^ong~V. L\^e, Yong-Geun Oh, Alfonso~G. Tortorella, and Luca Vitagliano,
  \emph{Deformations of coisotropic submanifolds in abstract {J}acobi
  manifolds}, preprint, 2014.

\bibitem[OP05]{OP2005}
Yong-Geun Oh and Jae-Suk Park, \emph{Deformations of coisotropic submanifolds
  and strong homotopy {L}ie algebroids}, Invent. Math. \textbf{161} (2005),
  no.~2, 287--360. \MR{2180451 (2006g:53152)}

\end{thebibliography}
\bibliographystyle{amsalpha}

\end{document}